\documentclass[hidelinks,onefignum,onetabnum]{siamart250211}

\usepackage{lipsum}
\usepackage{amsfonts}
\usepackage{bbm}
\usepackage{graphicx}
\usepackage{epstopdf}
\usepackage{algorithm,algpseudocode}
\usepackage{booktabs}
\usepackage{multirow}
\usepackage{longtable}
\ifpdf
  \DeclareGraphicsExtensions{.eps,.pdf,.png,.jpg}
\else
  \DeclareGraphicsExtensions{.eps}
\fi

\graphicspath{ {./images/} }

\newsiamremark{remark}{Remark}
\newsiamremark{hypothesis}{Hypothesis}
\crefname{hypothesis}{Hypothesis}{Hypotheses}
\newsiamthm{claim}{Claim}
\newsiamremark{fact}{Fact}
\crefname{fact}{Fact}{Facts}

\headers{Exploiting block triangular submatrices in KKT systems}{Robert Parker, Manuel Garcia, and Russell Bent}

\title{Exploiting block triangular submatrices in KKT systems
  \thanks{Submitted to the editors DATE.
\funding{Funded by the Los Alamos National Laboratory LDRD program. LA-UR-26-20941.}}}

\author{
  Robert Parker\thanks{Los Alamos National Laboratory, Los Alamos, NM} 
  \and Manuel Garcia\footnotemark[2] 
  \and Russell Bent\footnotemark[2]
}

\usepackage{amsopn}

\ifpdf
\hypersetup{
  pdftitle={Exploiting block triangular submatrices},
}
\fi

\begin{document}

\maketitle

\begin{abstract}
We propose a method for solving Karush-Kuhn-Tucker (KKT) systems that exploits
block triangular submatrices by first using a Schur complement decomposition
to isolate the block triangular submatrices then performing a block backsolve
where only diagonal blocks of the block triangular form need to be factorized.
We show that factorizing reducible symmetric-indefinite matrices with standard
1$\times$1 or 2$\times$2 pivots yields fill-in
outside the diagonal blocks of the block triangular form, in contrast to our
proposed method.
While exploiting a block triangular submatrix has limited fill-in, unsymmetric
matrix factorization methods do not reveal inertia, which is required by
interior point methods for nonconvex optimization.
We show that our target matrix has inertia that is known \textit{a priori}, letting
us compute inertia of the KKT matrix by Sylvester's law.
Finally, we demonstrate the computational advantage of this method on KKT systems
from optimization problems with neural network surrogates in their constraints.
Our method achieves up to 15$\times$ speedups over state-of-the-art symmetric
indefinite matrix factorization methods MA57 and MA86 in a constant-hardware
comparison.
\end{abstract}



\section{Introduction}
Factorizing symmetric, indefinite Karush-Kuhn-Tucker \linebreak(KKT) matrices is a
computational bottleneck in interior point methods for nonconvex, local optimization.
Methods for solving the KKT linear system can be grouped into three
broad categories:
\begin{enumerate}
  \item Direct methods,
  \item Iterative methods, and
  \item Decomposition methods (including Schur complement and null space decompositions).
\end{enumerate}
Interior point methods for nonconvex optimization most commonly use
direct methods for solving their KKT systems.
These methods factorize the matrix 
into the product $LBL^T$, where $L$ is lower triangular and $B$ is block-diagonal
with a combination of $1\times 1$ and $2\times 2$ blocks. These blocks correspond
to $1\times 1$ or $2\times 2$ pivots selected by the method of
Bunch and Kaufman \cite{bunch1977stable}.
The factors are typically computed by multifrontal \cite{duff1983multifrontal}
or supernodal \cite{liu1993supernodes} methods, which have been implemented
in several widely-used codes, e.g.,
MA57 \cite{ma57}, MA86 \cite{ma86}, MA97 \cite{ma97}, UMFPACK \cite{umfpack},
and others.

Iterative methods (e.g., Krylov subspace methods) have also been used
to solve KKT systems \cite{benzi2005saddle,curtis2010inexact}. These methods
do not require storing the matrix or its factors in memory simultaneously
and their repeated matrix-vector multiplications can be parallelized on GPUs.
However, the ill-conditioning encountered in interior point methods leads to slow
convergence \cite{swirydowicz2022linear}, making these methods less popular
in interior point implementations, especially for nonconvex optimization.
Nonetheless, iterative linear solvers for interior point methods are an active
area of research, with work focused on improving preconditioning
\cite{casacio2017preconditioning,shork2020preconditioning}
and analyzing the impact of inexact linear algebra \cite{zanetti2023krylov}.

Decomposition methods exploit the block structure of the KKT matrix,
\[\mathrm{KKT} = \left[\begin{array}{cc}A & B^T \\ B & \\ \end{array}\right],\]
to decompose the linear system solve into solves involving smaller systems.
These smaller systems may then be solved with either direct or iterative methods.
Null space methods are decompositions that exploit a partition of variables
into ``dependent variables'' and ``degrees of freedom'' in order to compute
a basis for the null space of $B$ \cite{benzi2005saddle,coleman1984}.
This basis is then used to reduce the KKT
system to a smaller system in the dimension of the number of degrees of freedom
(the number of variables minus the number of equality constraints).
These methods are typically effective if the number of degrees of freedom
is small, although constructing the null space basis can also be a bottleneck.
To overcome this bottleneck in the context of transient stability constrained
optimal power flow, Geng et al. \cite{geng2012parallel,geng2017gpu} have exploited a block triangular
decomposition for the unsymmetric solve required to construct the null space basis.

Schur complement methods exploit the nonsingularity of one or more symmetric
submatrices of the KKT matrix to decompose the linear system solve into
multiple smaller solves. Benzi et al. \cite{benzi2005saddle} describe a basic
Schur complement method exploiting the nonsingularity of $A$. A similar method
relying on regularization of $A$ has been described by Golub and Grief
\cite{golub2003} and implemented in the HyKKT solver \cite{regev2023hykkt}
using a direct method to construct the Schur complement and an iterative method
to solve the system involving the Schur complement matrix. In contrast to these
general-purpose Schur complement methods, many such methods have exploited
problem-specific structures to factorize KKT matrices in parallel.
For linear programs, Gondzio and Sarkissian \cite{gondzio2003parallel},
Grigoriadis and Khachiyan \cite{grigoriadis1996},
Lustig and Li \cite{lustig1992parallel},
and others have exploited decomposable block structures in interior
point methods.
In the context of nonconvex optimization, Wright \cite{wright1991partitioned}
and Jin et al. \cite{jin2025batched} have exploited nonsingular diagonal
blocks in block tridiagonal matrices, while
Laird and Biegler \cite{laird2008multiscenario} have exploited independent
submatrices due to distinct scenarios in stochastic optimization
and Kang et al. \cite{kang2014} have exploited nonsingular submatrices at
each time point in dynamic optimization.

Our contribution is a Schur complement method that, instead of exploiting
many independent, nonsingular submatrices, exploits a single nonsingular
submatrix that has a favorable structure for factorization and backsolves.
The structure we exploit is a block triangular structure (see \Cref{sec:blocktriangular})
that allows us to limit fill-in to diagonal blocks of this block triangular form,
which is not possible when factorizing these same systems with standard
direct methods.
Specifically, we:
\begin{enumerate}
  \item Propose a novel combination of symmetric Schur complement
    and unsymmetric block triangular decomposition algorithms,
  \item Demonstrate that standard $1\times 1$ and $2\times 2$ pivoting
    leads to fill-in outside the diagonal blocks of block triangular
    submatrices,
  \item Prove that regularization on the relevant block in our KKT
    matrices is both undesirable and unnecessary, and
  \item Demonstrate the computational advantage of our approach on
    KKT matrices derived from optimization problems with neural network
    constraints.
    (See \cite{omlt,schweidtmann2019deterministic,lopez2024} for background
    and applications of these optimization problems.)
\end{enumerate}
Our results show that exploiting one large submatrix with a special structure
can be advantageous in Schur complement methods for symmetric, indefinite
systems. While parallel computing is outside the scope of this work, we note
that exploiting a single large, nonsingular submatrix has potential to exploit
GPU acceleration as data transfer overhead is low compared to the potential
speedup that can be gained by processing this submatrix in parallel.

\section{Background}

\subsection{Interior point methods}

We consider symmetric indefinite linear systems from interior point methods
for solving nonlinear optimization problems in the form given by \Cref{eqn:nlopt}:
\begin{equation}
  \begin{array}{cll}
    \displaystyle\min_x & \varphi(x) \\
    \text{subject to} & f(x) = 0 \\
    & x \geq 0 \\
  \end{array}
  \label{eqn:nlopt}
\end{equation}
Here, $x$ is a vector of decision variables and functions $\varphi$ and $f$
are twice continuously differentiable.
We note that this formulation does not preclude general inequalities;
for our purposes, we assume they have already been reformulated with a slack variable.
Interior point methods converge to local solutions by updating an initial
guess $x^0$ using a search direction computed by solving the linear system
in \Cref{eqn:kkt}, referred to as the KKT system.
\begin{equation}
  \left[\begin{array}{cc}
      \nabla^2\mathcal{L} + \Sigma & \nabla f^T \\
      \nabla f & \\
  \end{array}\right]
  \left(\begin{array}{c}d_x \\ d_\lambda \\ \end{array}\right)
  = -\left(\begin{array}{c}
      \nabla \mathcal{L} \\
      f \\
  \end{array}\right)
  \label{eqn:kkt}
\end{equation}
Here, $\mathcal{L}$ is the Lagrangian function, $\lambda$ is the vector of Lagrange
multipliers of the equality constraints, and $\Sigma$ is a diagonal matrix containing
the ratio of bound multipliers, $\nu$, to variable values.
The matrix in \Cref{eqn:kkt}, referred to as the KKT matrix, is symmetric and indefinite.

\subsection{Inertia}
\label{sec:inertia}
Interior point methods usually require linear algebra subroutines that can not only
solve KKT systems but that can also compute the inertia of the KKT matrix.
\textit{Inertia} is defined as the numbers of positive, negative, and zero
eigenvalues of a symmetric matrix. Inertia is often revealed by symmetric
matrix factorization methods using Sylvester's Law of Inertia \cite{sylvester1852},
which states that two congruent matrices have the same inertia.
That is, if $A$ and $B$ are symmetric and $R$ is nonsingular, then
\begin{equation}
  A = R^T B R ~\Rightarrow~ \mathrm{Inertia}(A) = \mathrm{Inertia}(B)
  \label{eqn:sylvester}
\end{equation}

\subsection{Schur complement}
\label{sec:schur}
The Schur complement is the matrix obtained by pivoting on a nonsingular
submatrix. For a symmetric matrix
\[M = \left[\begin{array}{cc}A & B^T \\ B & C \\\end{array}\right]\]
with $C$ nonsingular, the Schur complement of $M$ with respect to $C$ is
\[S = A - B^T C^{-1} B\]
When solving the linear system $Mx = r$, or
\[\left[\begin{array}{cc}A & B^T \\ B & C \\\end{array}\right]
  \left(\begin{array}{c}x_A \\ x_C \\\end{array}\right)
  = \left(\begin{array}{c}r_A \\ r_C \\\end{array}\right),
\]
the Schur complement can be used to decompose the factorization
and backsolve phases into steps involving only $S$ or $C$, as shown
in \Cref{alg:schur-factorize,alg:schur-backsolve}.
\begin{algorithm}
  \caption{: Matrix factorization using the Schur complement}
  \begin{algorithmic}[1]
    \State {\bf Inputs:} Matrix $M = \left[\begin{array}{cc}A&B^T\\B&C\end{array}\right]$
    \State \texttt{factorize}$(C)$\label{alg:schur-factorize:fact-C}
    \State $S \gets A - B^T C^{-1} B$\label{alg:schur-factorize:build-S}
    \State \texttt{factorize}$(S)$
    \State {\bf Return:} \texttt{factors}$(C)$, \texttt{factors}$(S)$
  \end{algorithmic}
  \label{alg:schur-factorize}
\end{algorithm}
\begin{algorithm}
  \caption{: Backsolve using factors of the Schur complement}
  \begin{algorithmic}[1]
    \State {\bf Inputs:} \texttt{factors}$(C)$, \texttt{factors}$(S)$, right-hand side $r = (r_A, r_C)$
    \State $r_S \gets r_A - B^TC^{-1} r_C$
    \State $x_A \gets S^{-1} r_S$
    \State $\overline{r_C} \gets r_C - B r_A$
    \State $x_C \gets C^{-1} \overline{r_C}$
    \State {\bf Return:} $x = (x_A, x_C)$
  \end{algorithmic}
  \label{alg:schur-backsolve}
\end{algorithm}

The advantage of solving a system via this decomposition is that only matrices
$C$ and $S$ need to be factorized, rather than the full matrix $M$.
This is beneficial if the fill-in in $A$ due to $B^TC^{-1}B$ is relatively limited,
that is, if a majority of $B$'s columns are empty, or if $C$'s structure can be
exploited by a specialized algorithm.

When solving linear systems via a Schur complement, inertia is determined by the
Haynsworth inertia additivity formula \cite{haynsworth1968}:
\begin{equation}
  \mathrm{Inertia}\left(\left[\begin{array}{cc}
        A & B^T \\
        B & C \\
  \end{array}\right]\right)
  = \mathrm{Inertia}\left(C\right) + \mathrm{Inertia}\left(A - B^T C^{-1} B\right).
  \label{eqn:haynsworth}
\end{equation}

\subsection{Block triangular decomposition}
\label{sec:blocktriangular}
An $n\times n$ matrix $M$ is \textit{reducible} into block triangular form
if it may be permuted (by row and column permutations $P$ and $Q$)
to have the form
\begin{equation}
  PMQ = \left[\begin{array}{cc}M_{11} & 0 \\ M_{12} & M_{22} \\\end{array}\right],
  \label{eqn:decomposable}
\end{equation}
where $M_{11}$ is $k\times k$ and $M_{22}$ is $(n-k)\times(n-k)$ for some
$k$ between one and $n-1$.
That is, $M$ is reducible if there exists a permutation that reveals
a $k \times (n-k)$ zero submatrix.
(This definition is due to Frobenius \cite{frobenius1912}.
Some sources, e.g., \cite{pothen1990computing,davis2006}, use the equivalent
\textit{Strong Hall Property}.)
Here, the submatrices $M_{11}$ and $M_{22}$ may themselves be reducible.
The problem of decomposing a matrix into irreducible block triangular
form has been studied by Duff and Reid \cite{duff1978tarjan,duff1977permutations},
Gustavson \cite{gustavson1976}, Howell \cite{howell1976}, and others.
When solving a linear system of equations defined by a reducible matrix,
only the matrices on the diagonal need to be factorized.
For example, when solving the system $MX=B$, where $M$ decomposes into $b$
blocks, i.e.,
\[
\left[\begin{array}{ccc}M_{11} & & \\ \vdots & \ddots & \\ M_{1b}& \cdots & M_{bb} \\\end{array}\right]
\left(\begin{array}{c}X_1 \\ \vdots \\ X_b\\\end{array}\right)
= \left(\begin{array}{c}B_1 \\ \vdots \\ B_b \\ \end{array}\right),
\]
the following equation may be used to sequentially compute the blocks of the solution:
\begin{equation}
  X_i = M_{ii}^{-1} \left(B_i - \sum_{j=1}^{i-1}M_{ji} X_j\right) \text{ for } i \text{ in } \{1,\ldots,b\}.
  \label{eqn:bt-backsolve}
\end{equation}
The primary advantage of this procedure is that fill-in (nonzero entries that are present in
a matrix's factors but not the matrix itself) is limited to the submatrices along the diagonal,
$M_{ii}$, as these are the only matrices that are factorized.
By contrast, a typical $LU$ factorization of the matrix $M$ above may have fill-in
anywhere in or under the diagonal blocks. 
Additionally, L2 (matrix-vector) and L3 (matrix-matrix) BLAS are used directly in
the backsolve procedure, \Cref{eqn:bt-backsolve}. In typical multifrontal and supernodal
algorithms for symmetric-indefinite linear systems, performance benefits are obtained
by aggregating rows and columns into dense submatrices on which
L2 and L3 BLAS can efficiently operate. Exploiting a block triangular form (irreducible
or otherwise) gives us more control over the data structures used to handle the off-diagonal
submatrices. That is, these data structures may be sparse, dense, or some combination thereof
and can be chosen to optimize the relatively simple matrix subtraction and multiplication
operations in \Cref{eqn:bt-backsolve}, rather than chosen for performance in a more
complicated factorization algorithm.

\subsection{Graphs of sparse matrices}
The sparsity pattern of an $m\times n$ matrix may be represented as an undirected bipartite graph.
One set of nodes represents the rows and the other represents the columns. There exists
an edge $(i,j)$ if the matrix entry at row $i$ and column $j$ is structurally nonzero.

For a square, $n\times n$, matrix, a directed graph on a single set of nodes $\{1,\ldots,n\}$
may be defined. The graph has a directed edge $(i,j)$ if a structurally nonzero matrix
entry exists at row $i$ and column $j$.

These graphs are well-studied in the sparse matrix literature (see, e.g., \cite{davis2006}).
In this work, they are necessary only for the statement and proof of \Cref{thm:regularization}.

\subsection{Neural networks}
The methods we use and develop in this work are applicable to general block triangular
submatrices, including those found in multiperiod and discretized dynamic optimization problems.
Here, we demonstrate the benefits of our approach on block triangular submatrices
that arise in neural network-constrained optimization problems.
These optimization problems arise in design and control with neural network surrogates,
neural network verification, and adversarial example problems \cite{lopezflores2024process}.
We address KKT matrices derived from nonlinear optimization problems with neural network
constraints. A neural network, $y=\mathrm{NN}(x)$, is a function defined by repeated
application of an affine transformation and a nonlinear activation
function $\sigma$ over $L$ layers:
\begin{equation}
  \begin{aligned}
    y_l &= \sigma_l (W_l y_{l-1} + b_l) && l \in \{1,\dots,L\},
    \label{eqn:nn}
  \end{aligned}
\end{equation}
where $y_0 = x$ and $y = y_L$.
In this context, training weights $W_l$ and $b_l$ are considered constant.
We consider optimization problems with neural networks included in the constraints
using a full-space formulation \cite{schweidtmann2019deterministic,omlt,dowson2025moai}
in which intermediate variables and constraints are introduced for every layer
of the neural network. The neural network constraints in our optimization problem
are:
\begin{equation}
    \left.
    \begin{array}{l}
        y_l = \sigma_l(z_l) \\
        z_l = W_l y_{l-1} + b_l \\
    \end{array}
    \right\} \text{ for all } l\in \{1,\cdots,L\}
    \label{eqn:nn-full}
\end{equation}
We note that these neural network equations have a triangular
incidence matrix.
However, it is convenient to represent them in block triangular form,
shown in \Cref{fig:nn-incidence}, where blocks correspond to
layers of the neural network.
\begin{figure}[h]
  \centering
  \includegraphics[width=8cm]{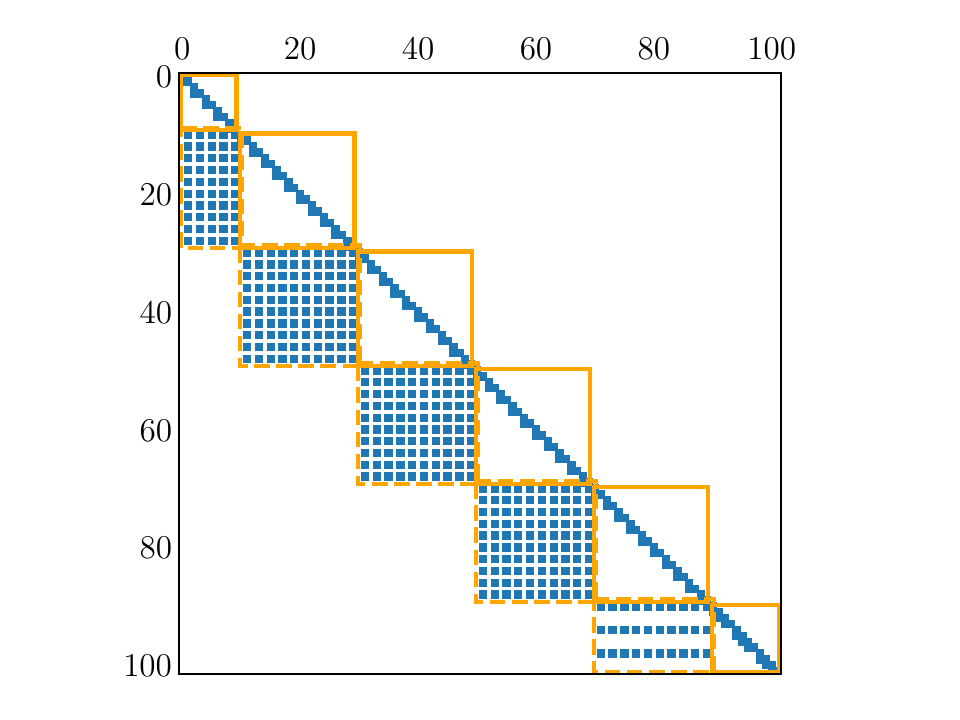}
  \caption{Incidence matrix structure of neural network constraints.
  Rows are equations and columns are output or intermediate variables.
  Each box along the diagonal contains the variables (and constraints)
  for a particular layer of the network, i.e., $y_l$, $z_l$, and the
  associated constraints from \Cref{eqn:nn-full}.
  The dashed boxes under the diagonal contain the nonzeros corresponding
  to the links between these variables and the variables of the next layer.
  That, is they contain the $W_l$ matrices from \Cref{eqn:nn-full}.
}
  \label{fig:nn-incidence}
\end{figure}

\section{Exploiting block triangularity within the Shur complement}

\subsection{Our approach}

We address KKT matrices derived from the following special case of \Cref{eqn:nlopt}:
\begin{equation}
  \begin{array}{cll}
    \displaystyle\min_{x,y} & \varphi(x,y) \\
    \text{subject to} & f(x,y) = 0 \\
    & g(x,y) = 0\\
    & L \leq (x, y) \leq U
  \end{array}
  \label{eqn:nlopt-xy}
\end{equation}
Here, we assume that $\nabla_y g$ is nonsingular and block-lower triangular
(with square, nonsingular blocks along the diagonal).
Let the row and column dimensions of $\nabla_y g$ be $n_y$.
The KKT system of \Cref{eqn:nlopt-xy} is:
\begin{equation}
  \left[\begin{array}{cc|cc}
      W_{xx} & \nabla_x f^T & W_{xy} & \nabla_x g^T \\
      \nabla_x f & & \nabla_y f & \\[2pt]
      \hline
      W_{yx} & \nabla_y f^T & W_{yy} & \nabla_y g^T \\
      \nabla_x g & & \nabla_y g & \\
  \end{array}\right]
  \left(\begin{array}{c}
      d_x \\
      d_{\lambda_f} \\
      \hline
      d_y \\
      d_{\lambda_g} \\
  \end{array}\right)
  = -\left(\begin{array}{c}
      \nabla_x \mathcal{L} \\
      f \\
      \hline
      \nabla_y \mathcal{L} \\
      g \\
    \end{array}\right)
  \label{eqn:kkt-xy}
\end{equation}
Partitioning \Cref{eqn:kkt-xy} as shown above, we rewrite
the system as \Cref{eqn:kkt-partition} for simplicity.
\begin{equation}
  \left[\begin{array}{cc}
      A & B^T \\
      B & C \\
  \end{array}\right]
  \left(\begin{array}{c}
      d_A \\
      d_C \\
  \end{array}\right)
  = \left(\begin{array}{c}
      r_A \\
      r_C \\
    \end{array}\right)
  \label{eqn:kkt-partition}
\end{equation}
We will solve KKT systems of this form by performing a Schur complement reduction
with respect to submatrix $C$, which we refer to as the \textit{pivot matrix}.
Because $\nabla_y g$ is block-lower triangular, $C$
may be permuted to have block-lower triangular form as well.
Specifically, the following permutation vectors permute $C$ to block
lower triangular form:
\begin{equation}
  \begin{array}{rrl}
    \text{Column permutation:} & \left\{1,\ldots,n_y,\right.& \hspace{-0.3cm}\left. 2n_y,\ldots,n_y+1\right\}\\
    \text{Row permutation:} & \left\{n_y+1,\ldots,2n_y,\right. & \hspace{-0.3cm}\left. n_y,\ldots,1\right\}.\\
  \end{array}
  \label{eqn:permutation}
\end{equation}
These permutation vectors are defined such that the row or column index at the $i$-th
position of the permutation vector is the index of the row or column in the original
matrix that is permuted to position $i$ in the new matrix.
This permutation is illustrated with permutation matrices $P$ and $Q$ in \Cref{fig:permutation}.
\begin{figure}[h]
  \centering
  \includegraphics{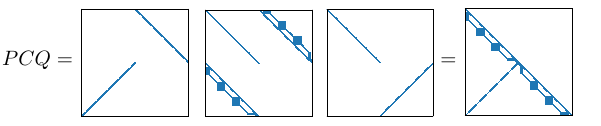}
  \caption{Permutation of $C$ to block triangular form}
  \label{fig:permutation}
\end{figure}

We solve \Cref{eqn:kkt-partition} using a Schur complement with respect
to $C$ via \Cref{alg:schur-factorize,alg:schur-backsolve}.
The computational bottlenecks of this procedure are factorizing the
pivot matrix, $C$, performing many 
backsolves with this matrix to construct the Schur complement matrix, $S$, and
factorizing this Schur complement matrix.
We perform factorizations and backsolves that exploit $C$'s block triangular structure
as described in \Cref{sec:blocktriangular}.

\subsection{Application to neural network submatrices}

The above methods apply to general block triangular Jacobian submatrices
$\nabla_y g$. However, our motivation for developing this method is to exploit
the structure of a neural network's Jacobian. We now discuss the advantages
that this particular structure gives us.

To exploit the structure of neural network constraints, we let
constraint function $g$ (the function whose structure we exploit) be the
functions in \Cref{eqn:nn}. The variables we exploit are the outputs of the
neural network and intermediate variables defined at each layer:
$\{z_1,y_1,\ldots,z_L,y_L\}$.
The Jacobian matrix has the following structure:
\footnote{The structures shown in \Cref{fig:nn-incidence} and
  \Cref{eqn:nn-jac} differ due to different orderings of variables
and equations.}
\begin{equation}
  \left[\begin{array}{cccccc}
    I \\
    \Sigma_1 & I \\
    & W_2 & \ddots \\
    & & \ddots & I \\
    & & & W_L & I \\
    & & & & \Sigma_L & I \\
  \end{array}\right].
  \label{eqn:nn-jac}
\end{equation}
Here, $\Sigma_i$ is the diagonal matrix containing entries $\sigma_i'(z_i)$,
where $\sigma_i'$ is the first derivative of the activation function $\sigma_i$, applied
elementwise. This Jacobian matrix (\ref{eqn:nn-jac}) is lower triangular.
However, when exploiting this structure, it is advantageous to consider blocks each
corresponding to a layer of the neural network, as shown. In this aggregation,
diagonal blocks are themselves diagonal and off-diagonal blocks are either diagonal
matrices containing derivatives of activation functions or the dense weight matrices
that serve as inter-layer connections.
Because the diagonal blocks are identity matrices, they do not need to be factorized
and the block triangular backsolve procedure, \Cref{eqn:bt-backsolve}, will incur
no fill-in. Furthermore, the off-diagonal matrices, $M_{ji}$ in \Cref{eqn:bt-backsolve},
are either diagonal or fully dense.
We store these dense matrices as contiguous arrays (rather than with a COO or CSC
sparse matrix data structure) to take advantage of efficient L3 BLAS operations.

Finally, we observe that for the case of a neural network's Jacobian,
\Cref{eqn:bt-backsolve} is very similar to the forward pass through the network:
$X_j$ is the solution for a previous layer's intermediate variables, $M_{ji}$
is the inter-layer weight matrix, and the backsolve with $M_{ii}$ takes the place
of an activation function.


\subsection{Fill-in}
\label{sec:fillin}
The core benefit of our proposed algorithm is that we do not have to factorize the pivot
matrix, $C$, directly. Instead, we factorize only the diagonal blocks of $C$'s block triangular
form, $C_{ii}$, and therefore only incur fill-in in these blocks.
As we will show, this procedure incurs significantly less fill-in than we expect
when factorizing the symmetric matrix $C$ directly.
Constructing the Schur complement matrix $S$, however, incurs fill-in in the columns
of $B$ that are nonzero.\footnote{See the dense blocks in the Schur complement matrices
in \Cref{fig:mnist-spy,fig:scopf-spy,fig:lsv-spy}. In our neural network application,
these dense blocks correspond to the neural network's inputs and the constraints
containing any output of the network.}
Therefore, our method is particularly effective when the
number of nonzero columns of $B$, i.e., the number of variables and constraints
``linking'' $A$ and $C$, is small.
We will now show that $1\times 1$ and $2 \times 2$ pivots used by traditional
symmetric indefinite linear solvers do not capture the structure of $C$ and lead
to fill-in outside of its diagonal blocks.

For simplicity, we rewrite $C$ in its symmetric form as
\[C = \left[\begin{array}{cc}D & E^T \\ E & \end{array}\right],\]
then partition $D$ and $E$:
\[D = \left[\begin{array}{ccc}
      d_{11} & D^T_{21} & D^T_{31} \\
      D_{21} & D_{22} & D^T_{32} \\
      D_{31} & D_{32} & D_{33} \\
  \end{array}\right]\text{ and }~E = \left[\begin{array}{ccc}
      e_{11} & E_{12} & E_{13}\\
      E_{21} & E_{22} & E_{23}\\
      & & E_{33} \\
  \end{array}\right].
\]
Here, $d_{11}$ and $e_{11}$ are scalars corresponding to a pivot and
$\left[\begin{array}{cc}e_{11} & E_{12} \\ E_{21} & E_{22}\end{array}\right]$
and $E_{33}$ are two square diagonal blocks in a block triangular form of $E$.
Submatrix $E_{22}$ has the rows and columns that are in the same diagonal
block as our pivot while $E_{33}$ has row and column indices that are in
a different diagonal block.
With our proposed algorithm, 
fill-in is limited to these diagonal blocks. That is, our algorithm incurs
no fill-in in $E_{13}$, $E_{23}$, their transposes, or $D$.
In the rest of this section, we will show that $1\times 1$ or $2\times 2$
pivots lead to fill-in in $D$ and in the $E_{23}$ and $E_{23}^T$ blocks.

Applying permutation matrix $P$ to put a potential $2 \times 2$ pivot involving
$e_{11}$ in the upper left of the matrix,
\begin{equation}
  PCP^T = \left[\begin{array}{cccccc}
      d_{11} & e_{11} & D^T_{21} & E_{21}^T & D^T_{31} & \\
      e_{11} & & E_{12} & & E_{13}\\
      D_{21} & E_{12}^T & D_{22} & E_{22}^T & D^T_{32} & \\
      E_{21} & & E_{22} & & E_{23}\\
      D_{31} & E_{13}^T & D_{32} & E_{23}^T & D_{33} & E_{33}^T\\
      & & & & E_{33} \\
  \end{array}\right].
  \label{eqn:c-partition}
\end{equation}
We will show that $1 \times 1$ and $2 \times 2$ pivots may cause fill-in in $D$ and
in off-diagonal blocks of $E$ (i.e., in $E_{23}$), in contrast to our proposed method.
\begin{theorem}
  Let $C$ be the matrix in the right-hand side of \Cref{eqn:c-partition}.
  A $1\times 1$ pivot on $d_{11}$ may result in fill-in everywhere except
  the last block row and column.
  \label{thm:1by1}
\end{theorem}
\begin{proof}[Proof of \Cref{thm:1by1}]
  After pivoting on $d_{11}$, the remainder of the matrix is the Schur complement:
  \[\left[\begin{array}{ccccc}
      & E_{12} & & E_{13}\\
      E_{12}^T & D_{22} & E_{22}^T & D^T_{32} & \\
      & E_{22} & & E_{23}\\
      E_{13}^T & D_{32} & E_{23}^T & D_{33} & E_{33}^T\\
      & & & E_{33} \\
    \end{array}\right]
    - \left[\begin{array}{c} e_{11} \\ D_{21} \\ E_{21} \\ D_{31} \\ 0 \\ \end{array}\right]
    \frac{1}{d_{11}} 
    \left[\begin{array}{ccccc} e_{11} & D_{21}^T & E_{21}^T & D_{31}^T & 0 \end{array}\right]
    .
  \]
  Expanding the second term, this becomes:
  \[
      \left[\begin{array}{ccccc}
      & E_{12} & & E_{13}\\
      E_{12}^T & D_{22} & E_{22}^T & D^T_{32} & \\
      & E_{22} & & E_{23}\\
      E_{13}^T & D_{32} & E_{23}^T & D_{33} & E_{33}^T\\
      & & & E_{33} \\
    \end{array}\right]
    - \left[\begin{array}{ccccc}
        e_{11}^2 & e_{11}D_{21}^T & e_{11}E_{21}^T & e_{11}D_{31}^T & 0\\
        D_{21}e_{11} & D_{21}D_{21}^T & D_{21}E_{21}^T & D_{21}D_{31}^T & 0\\
        E_{21}e_{11} & E_{21}D_{21}^T & E_{21}E_{21}^T & E_{21}D_{31}^T & 0\\
        D_{31}e_{11} & D_{31}D_{21}^T & D_{31}E_{21}^T & D_{31}D_{31}^T & 0\\
        0 & 0 & 0 & 0 & 0 \\
  \end{array}\right]
  \]
  Any nonzero entries in the second term that are not in the first term are fill-in.
  These entries may exist in any of the nonzero blocks of the second term,
  which populate the entire matrix other than the last block row and column.
\end{proof}
\begin{theorem}
  Let $C$ be the matrix in the right-hand side of \Cref{eqn:c-partition}.
  A $2\times 2$ pivot on
  $\left[\begin{array}{cc}d_{11} & e_{11} \\ e_{11} & \\\end{array}\right]$
  results in fill-in in the $D_{22}$, $E_{22}$, $D_{32}$, $E_{23}$, and $D_{33}$
  blocks and their transposes, if applicable.
  \label{thm:2by2}
\end{theorem}
\begin{proof}[Proof of \Cref{thm:2by2}]
  After pivoting on $\left[\begin{array}{cc}d_{11} & e_{11} \\ e_{11}\end{array}\right]$,
  the remainder of $C$ is the Schur complement:
  \[\left[\begin{array}{cccc}
      D_{22} & E_{22}^T & D^T_{32} & \\
      E_{22} & & E_{23}\\
      D_{32} & E_{23}^T & D_{33} & E_{33}^T\\
      & & E_{33} \\
    \end{array}\right]
    - \left[\begin{array}{cc}
        D_{21} & E_{12}^T \\
        E_{21} & 0 \\
        D_{31} & E_{13}^T \\
        0 & 0 \\
    \end{array}\right]
    \left[\begin{array}{cc}
        & \frac{1}{e_{11}} \\
        \frac{1}{e_{11}} & -\frac{d_{11}}{e_{11}^2} \\
    \end{array}\right]
    \left[\begin{array}{cccc}
        D_{21}^T & E_{21}^T & D_{31}^T & 0 \\
        E_{12} & 0 & E_{13} & 0 \\
    \end{array}\right]
  \]
  Expanding the second term, this becomes:
  \[
    \scalebox{0.9}{$
      \frac{1}{e_{11}}
    \left[\begin{array}{cccc}
      \left(E_{12}^TD_{21}^T + D_{21}E_{12} - \frac{d_{11}}{e_{11}} E_{12}^TE_{12}\right)
        & E_{12}^TE_{21}^T
        & \left(E_{12}^TD_{31}^T + D_{21}E_{13} - \frac{d_{11}}{e_{11}}E_{13}^TE_{13}\right)
        & 0 \\
      E_{21}E_{12} & 0 & E_{21}E_{13} & 0 \\
      \left(E_{13}^TD_{21}^T + D_{31}E_{12} - \frac{d_{11}}{e_{11}}E_{13}^T E_{12}\right)
       & E_{13}^T E_{21}^T
       & \left(E_{13}^TD_{31}^T + D_{31}E_{13} - \frac{d_{11}}{e_{11}}E_{13}^TE_{13}\right)
       & 0 \\
      0 & 0 & 0 & 0 \\
    \end{array}\right]
$}
\]
  Any nonzero entries in the second term that are not in the first term are
  fill-in. These entries may exist in any of the nonzero blocks of the second
  term, which are limited to the $D_{22}$, $E_{22}$, $D_{23}$, $E_{23}$, and
  $D_{33}$ blocks of the original matrix (and their transposes).
\end{proof}

Theorems \ref{thm:1by1} and \ref{thm:2by2} apply to pivots induced by permuting
$E$ to block upper-triangular form.
A similar result holds when $E$ is initially in block-\textit{lower} triangular form.
We do not claim that either of these will
result in a minimum-fill ordering or even that our results prove that a
$2\times 2$ pivot is always strictly better than a $1\times 1$ pivot, as this will
depend on sparsity patterns of the submatrices.
However, these block-fill-in analyses give us reason to prefer a $2\times 2$
pivot involving $e_{11}$ to a $1\times 1$ pivot. Our experience,
which we justify with results in \Cref{sec:num-results}, is that
traditional symmetric indefinite linear solvers prioritize $1\times 1$ pivots.
We hypothesize that this is one reason for the large fill-in we will observe.
We again highlight that neither $1\times 1$ nor $2\times 2$ pivots obtain the
favorable result of our specialized method that fill-in is limited to diagonal
blocks of $E$.

\subsection{Inertia and regularization}
Interior point methods require linear algebra subroutines that can not only solve
KKT systems but that can also compute the inertia of the KKT matrix.
As discussed in \Cref{sec:schur}, Schur complement methods typically use the
Haynsworth formula to check inertia given the inertia of the factorized pivot
and Schur complement matrices.
However, our block triangular factorization of $C$ does not reveal its
inertia. Fortunately, our assumption that $\nabla_y g$ is nonsingular implies
that the inertia of $C$ is $(n_y, n_y, 0)$ via \Cref{thm:inertia}
(a special case of Lemma 3.2 in \cite{gould1985}).
\begin{theorem}
  Let $C$ be as defined in \Cref{eqn:kkt-partition}:
  \[C = \left[\begin{array}{cc}
      W_{yy} & \nabla_y g^T \\
      \nabla_y g & \\
  \end{array}\right],
\]
where $\nabla_y g$ is $n_y\times n_y$ and nonsingular.
The inertia of $C$ is $(n_y, n_y, 0)$.
  \label{thm:inertia}
\end{theorem}
\begin{proof}[Proof of \Cref{thm:inertia}]
  Let
  \[R = \left[\begin{array}{cc} I & -\frac{1}{2}W_{yy}\nabla_y g^{-1} \\ & \nabla_y g^{-1} \end{array}\right].\]
  Then
  \[RCR^T = \left[\begin{array}{cc} & I \\ I & \end{array}\right].\]
  The matrix $\left[\begin{array}{cc} & I \\ I & \end{array}\right]$, where $I$ is $n_y \times n_y$,
  has inertia $(n_y, n_y, 0)$. 
  Then, because $R$ is nonsingular, by Sylvester's Law of Inertia, $C$ has inertia
  $(n_y, n_y, 0)$ as well.
\end{proof}

The inertia of $C$ is therefore known \textit{a priori}, so as long as an
inertia-revealing method (i.e., an $LBL^T$ factorization) is used for the
factorization of the Schur complement matrix, we can recover the inertia of the
original matrix as well.

Interior point methods use inertia to check conditions required for global convergence.
In particular, the search direction computed must be a descent direction for the
barrier objective function when projected into the null space of the equality constraint
Jacobian \cite{waechter2005linesearch}. 
This condition is met if the inertia of the KKT matrix is $(n,m,0)$, where $n$ is the number
of variables and $m$ is the number of equality constraints.
If the KKT matrix is factorized and the inertia is incorrect, a multiple of the identity
matrix is added to the $(1,1)$ or $(2,2)$ block of the KKT matrix, depending on whether
more positive or negative eigenvalues are needed.

While inertia correction is necessary for convergence of interior point methods,
adding multiples of the identity matrix to $C$ causes it to lose its property
of reducibility. This is formalized in \Cref{thm:regularization}.

\begin{theorem}
  Let $S$ be the following symmetric, block-$2\times2$ matrix:
  \[S = \left[\begin{array}{cc} I & B^T \\ B & I \\ \end{array}\right].\]
  Let $G_B$ be the bipartite graph of $B$. If $G_B$ is connected, then
  $S$ is irreducible.
  \label{thm:regularization}
\end{theorem}
We use the following theorem (Varga \cite{varga2000}, Theorem 1.17):
\begin{theorem}
  Let $M$ be an $n\times n$ matrix and $G$ its directed graph.
  $M$ is irreducible if and only if $G$ is strongly connected.
  \label{thm:varga}
\end{theorem}
\begin{proof}[Proof of \Cref{thm:regularization}]
  By \Cref{thm:varga}, it is sufficient to show that the directed graph of $S$,
  $G_S$, is strongly connected.
  Let $n$ be the row and column dimension of $B$. The lower-left block,
  $B$, adds edges from $\{1,\ldots,n\}$ to $\{n+1,\ldots,2n\}$ (corresponding
  to nonzero entries of $B$) to $G_S$
  while the upper-right block, $B^T$, adds the reverse edges.
  We compress these pairs of edges to form an undirected graph,
  where connectivity in the undirected graph is equivalent to strong connectivity
  in the directed graph.
  This undirected graph is identical to $B$'s bipartite graph, $G_B$, with the following node mapping:
  Nodes $\{1,\ldots,n\}$ in $G_S$ map to column nodes in $G_B$ while $\{n+1,\ldots,2n\}$
  in $G_S$ map to row nodes in $G_B$.
  Therefore connectivity of $G_B$ implies strong connectivity of $G_S$, which (by \Cref{thm:varga})
  is equivalent to irreducibility of $S$.
\end{proof}

The regularized form of our pivot matrix, denoted $\tilde C$, is:
\[\tilde C = \left[\begin{array}{cc}
      W_{yy} + \delta_y I & \nabla_y g^T \\
      \nabla_y g & \delta_g I \\
  \end{array}\right].
\]
Submatrix $\nabla_y g$ is the Jacobian of neural network constraints
and has a connected bipartite graph. Then by \Cref{thm:regularization}, $\tilde C$ does
not decompose via block triangularization.
Therefore, we do not add regularization terms in the pivot matrix. We
assume $\nabla_y g$ is nonsingular, so regularization is not necessary to correct
zero eigenvalues.

\section{Test problems}

We test our method on KKT matrices from three optimization problems involving neural network
constraints: Adversarial generation for an MNIST image classifier (denoted ``MNIST''), security-constrained
optimal power flow with transient feasibility encoded by a neural network surrogate (denoted ``SCOPF''),
and load shed verification of power line switching decisions made by a neural network (denoted ``LSV'').
Structural properties of the neural networks and resulting optimization problems are shown in
\Cref{tab:nn,tab:structure}.

\begin{table}[h]
  \centering
  \small
  \caption{Properties of neural network models used in this work}
  \resizebox{\textwidth}{!}{
  \begin{tabular}{ccccccc}
\toprule
Model & N. inputs & N. outputs & N. neurons & N. layers & N. param. &      Activation \\
\midrule
\multirow{3}{*}{MNIST} & \multirow{3}{*}{784} & \multirow{3}{*}{10} &         3k &         7 &        1M &    Tanh+SoftMax \\
                       &                      &                     &         5k &         7 &        5M &    Tanh+SoftMax \\
                       &                      &                     &        11k &         7 &       18M &    Tanh+SoftMax \\
\midrule
\multirow{3}{*}{SCOPF} & \multirow{3}{*}{117} & \multirow{3}{*}{37} &         1k &         5 &      578k &            Tanh \\
                       &                      &                     &         5k &         7 &        4M &            Tanh \\
                       &                      &                     &        12k &        10 &       15M &            Tanh \\
\midrule
\multirow{3}{*}{LSV} & \multirow{3}{*}{423} & \multirow{3}{*}{186} &        993 &         5 &      111k &         Sigmoid \\
                     &                      &                      &         2k &         5 &      837k &         Sigmoid \\
                     &                      &                      &         6k &         5 &        9M &         Sigmoid \\
\bottomrule
  \end{tabular}
}
  \label{tab:nn}
\end{table}

\begin{table}[h]
  \centering
  \caption{Structure of optimization problems with different neural networks embedded}
  \begin{tabular}{cccccc}
\toprule
Model & N. param. & N. var. & N. con. & Jac. NNZ & Hess. NNZ \\
\midrule
\multirow{3}{*}{MNIST} &        1M &      7k &      5k &       1M &        2k \\
                       &        5M &     12k &     11k &       5M &        5k \\
                       &       18M &     22k &     21k &      18M &       10k \\
\midrule
\multirow{3}{*}{SCOPF} &      578k &      4k &      4k &     567k &        9k \\
                       &        4M &     11k &     11k &       4M &       12k \\
                       &       15M &     25k &     25k &      15M &       19k \\
\midrule
\multirow{3}{*}{LSV}   &      111k &      8k &      4k &     125k &        4k \\
                       &      837k &     10k &      6k &     854k &        5k \\
                       &        9M &     19k &     15k &       9M &       10k \\
\bottomrule
  \end{tabular}
  \label{tab:structure}
\end{table}

\subsection{Adversarial image generation (MNIST)}

The MNIST set of handwritten digits \cite{lecun1998mnist} is a set of images 
of handwritten digits commonly used in machine learning benchmarks.
We have trained a set of neural networks to classify these images using PyTorch \cite{paszke2019pytorch}
and the Adam optimizer \cite{adam}. The networks each accept a 28$\times$28 grayscale
image and output scores for each digit, 0-9, that may be interpreted as a probability
of the image being that digit. The total numbers of neurons and trained parameters in the three
networks we test for this problem are shown in \Cref{tab:nn}.
After training these neural networks, our goals is to find a minimal perturbation
to a reference image that results in a particular misclassification.
This optimization problem is given by \Cref{eqn:adversarial-image},
which is inspired by the minimal-adversarial-input formulations of
Modas et al. \cite{modas2019sparse} and Croce and Hein \cite{croce2020minimal}.
\begin{equation}
  \begin{array}{cll}
    \displaystyle\min_x & \left\|x - x_{\rm ref}\right\|_1 \\
    \text{subject to} & y = \mathrm{NN}(x) \\
    & y_t \geq 0.6 \\
  \end{array}
  \label{eqn:adversarial-image}
\end{equation}
Here, $x\in \mathbb{R}^{784}$ contains grayscale values of the generated image
and $x_{\rm ref}$ contains those of the reference image. Vector $y\in\mathbb{R}^{10}$
contains the output score of each digit and $t$ is the coordinate corresponding to the
(misclassified) target label of the generated image. We note that $t$ is a constant
parameter chosen at the time we formulate the problem.
We enforce that the new image is classified as the target with a probability of at least
60\%. The sparsity structures of this problem's KKT, pivot, and Schur complement matrices
are shown in \Cref{fig:mnist-spy}.
\begin{figure}[h]
  \centering
  \includegraphics[width=\textwidth]{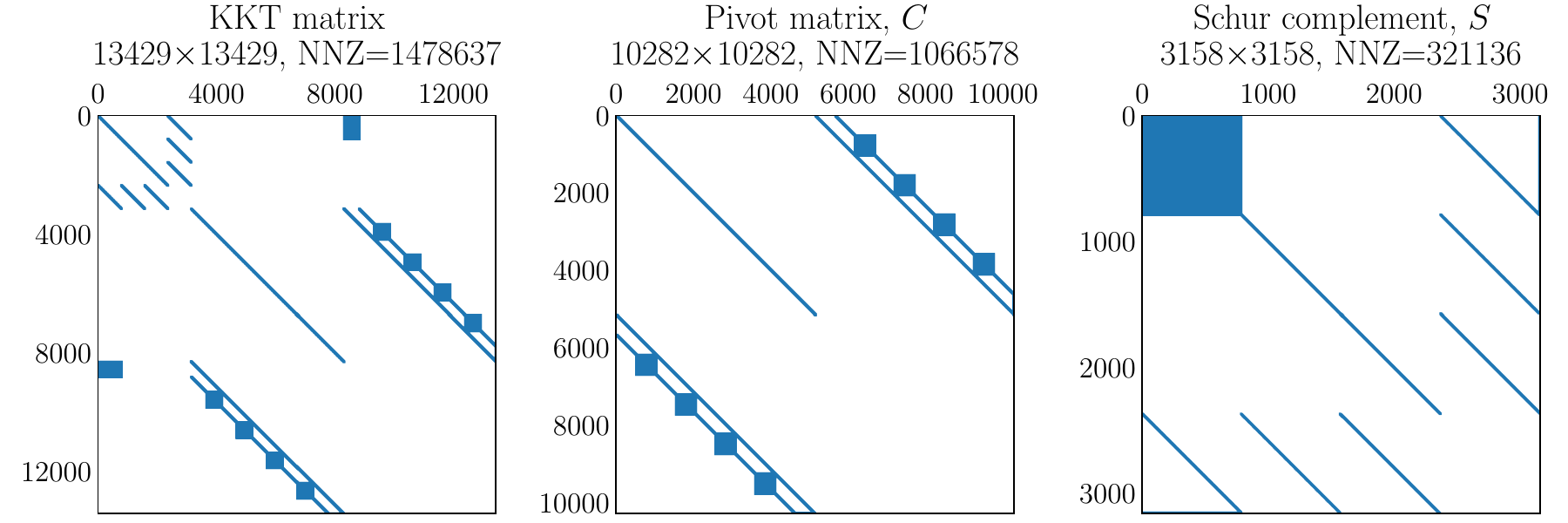}
  \vspace{-0.7cm}
  \caption{Structure of KKT, pivot, and Schur complement matrices for the smallest
  instance of the MNIST test problem. The KKT matrix is in the order shown in \Cref{eqn:kkt-xy}.}
  \label{fig:mnist-spy}
\end{figure}

\subsection{Security-constrained AC optimal power flow (SCOPF)}
Security-constrained alternating current (AC) optimal power flow (SCOPF)
is an optimization problem for dispatching generators in an electric power grid
in which feasibility of the grid (with respect to voltage, line flow, and load satisfaction
constraints) is enforced for a set of \textit{contingencies} $k\in K$ \cite{aravena2023}.
Each contingency represents the loss of a set of generators and/or power lines.
Instead of enforcing feasibility at the grid's new steady state, we enforce feasibility
of the first 30 seconds of the transient response. Specifically, we enforce that
the frequency at each bus remains above $\eta = 59.4~\mathrm{Hz}$.
Instead of inserting differential equations describing transient grid behavior
into the optimization problem (as done by, e.g., Gan et al. \cite{gan2000}),
we use a neural network that predicts the minimum frequency over this horizon.
A simplified formulation is given in \Cref{eqn:scopf}:
\begin{equation}
  \begin{array}{cll}
    \displaystyle\min_{S^g,V} & c(\mathbb{R}(S^g)) \\
    \text{subject to} & F_k(S^g, V, S^d) \leq 0 & k\in \{0,\ldots,K\} \\
    & \left.\begin{array}{l}y_k = \mathrm{NN}_k(x) \\ y_k\geq\eta \\ \end{array} \right\} 
      & k\in \{1,\ldots,K\}\\
  \end{array}
  \label{eqn:scopf}
\end{equation}
Here, $S^g$ is a vector of complex AC power generations for each generator in the network,
$V$ is a vector of complex bus voltages, $c$ is a quadratic cost function, and $S^d$ is a constant
vector of complex power demands. Constraint $F_k\leq 0$ represents the set of constraints enforcing
feasibility of the power network for contingency $k$ (see \cite{cain2012history}), where $k=0$
refers to the base-case network, and $\mathrm{NN}_k$ is the neural network
function predicting minimum frequency following contingency $k$.
This formulation is inspired by that of Garcia et al. \cite{garcia2025transient}.
We consider a simple instance of \Cref{eqn:scopf} defined on a 37-bus synthetic test grid
\cite{birchfield2017,hawaii40} with a single contingency outaging generator 5 on bus 23.
Here, the neural network has 117 inputs and 37 outputs.
This surrogate is trained on data from 110 high-fidelity simulations using PowerWorld \cite{powerworldmanual}
with generations and loads uniformly sampled from within a $\pm20\%$ interval of each nominal
value.
These neural networks have between five and ten layers and
between 500 and 1,500 neurons per layer (see \Cref{tab:nn}).
The sparsity structures of this problem's KKT, pivot, and Schur complement matrices are shown
in \Cref{fig:scopf-spy}.
\begin{figure}[h]
  \centering
  \includegraphics[width=\textwidth]{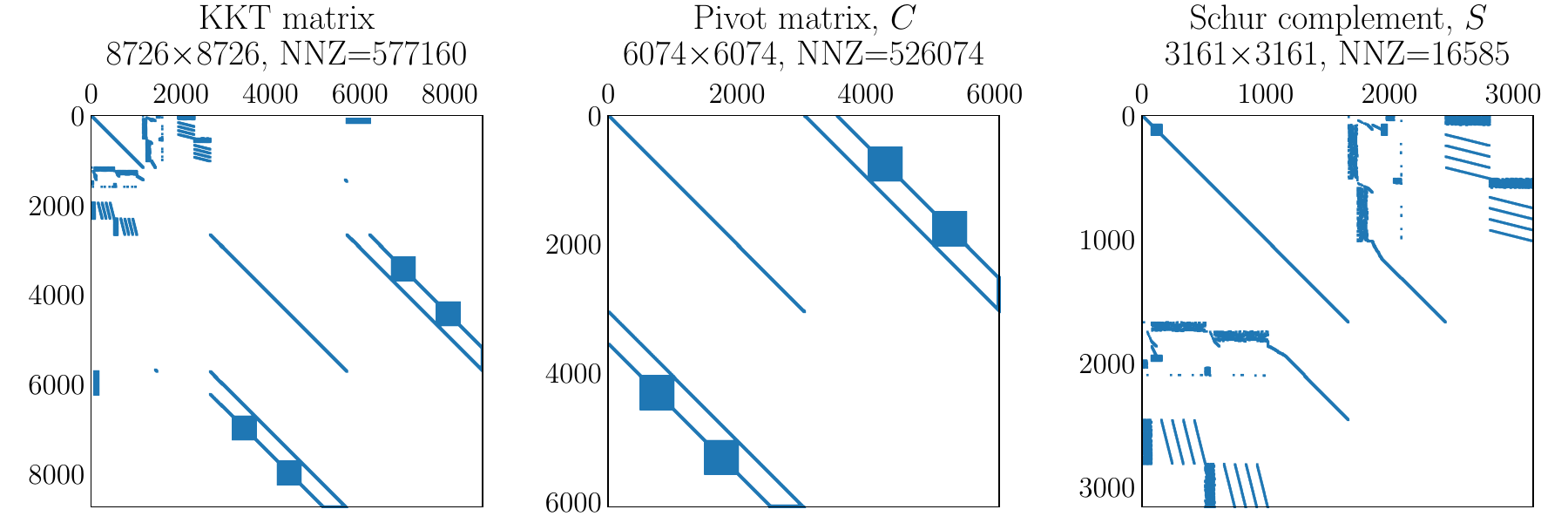}
  \vspace{-0.7cm}
  \caption{Structure of KKT, pivot, and Schur complement matrices for the smallest
  instance of the SCOPF test problem.}
  \label{fig:scopf-spy}
\end{figure}

\subsection{Line switching verification (LSV)}
Our final test problem maximizes load shed due to line switching decisions made
by a neural network that accepts loading conditions and wildfire risk (for every
bus in a power network) as inputs. This is Problem 3 presented by Chevalier et al.
\cite{chevalier2025}.
The goal of the problem in \cite{chevalier2025} is to maximize load shed due to
neural network switching decisions and an operator that acts to minimize a combination of
load shed and wildfire risk. This bilevel problem is given by \Cref{eqn:bilevel-lsv}.
\begin{equation}
  \begin{array}{cll}
    \displaystyle\max_{x,y,p_g} & C_1 (x, y, p_g) \\
    \text{subject to} & y = \mathrm{NN}(x) \\
    & p_g \in \underset{p_g\in\mathcal{F}(x,y)}{\text{argmin}} C_2 (x, y, p_g) \\
  \end{array}
  \label{eqn:bilevel-lsv}
\end{equation}
Here, both inner and outer problems are non-convex. The feasible set of the inner problem,
$\mathcal{F}(x,y)$, is defined by the power flow equations with fixed loading, $x$, and
line switching decisions, $y$.
To approximate this problem with a single-level optimization problem, we follow
\cite{chevalier2025} and relax the inner problem with a second-order cone relaxation
\cite{jabr2006}, dualize the (now convex) inner problem, and combine the two levels
of the optimization to yield a single-level nonconvex maximization problem.
We refer the reader to \cite{chevalier2025} for further details, including a description
of how these neural networks are trained.
\begin{figure}[h]
  \centering
  \includegraphics[width=\textwidth]{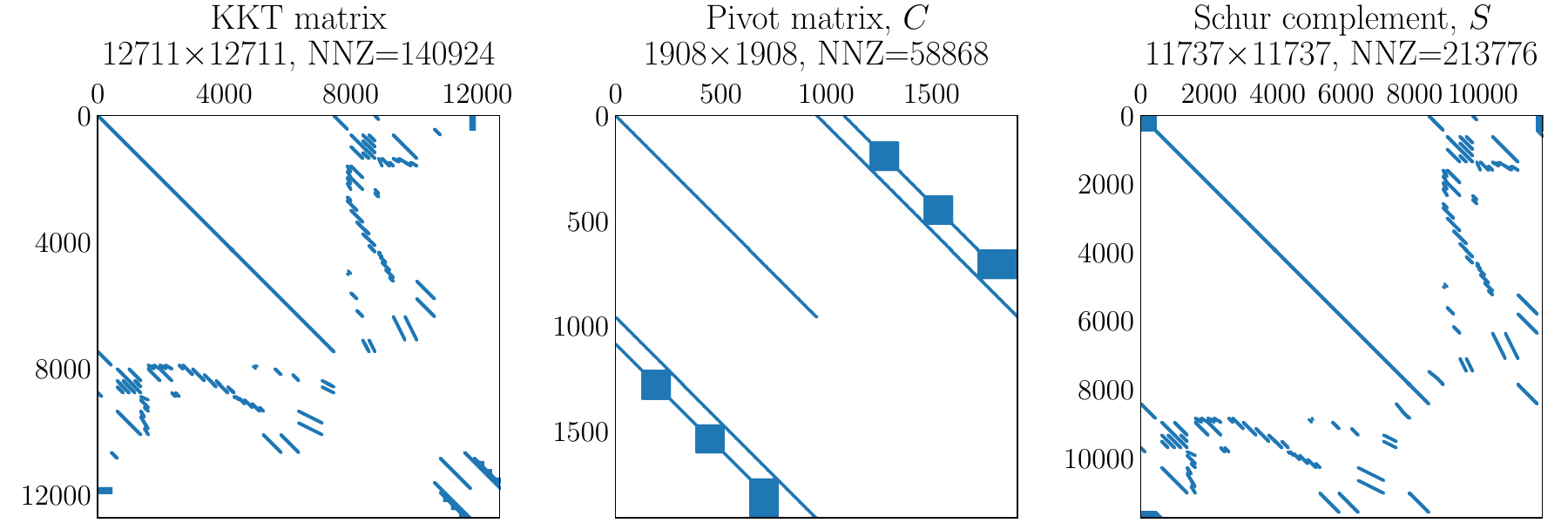}
  \vspace{-0.7cm}
  \caption{Structure of KKT, pivot, and Schur complement matrices for the smallest
  instance of the LSV test problem.}
  \label{fig:lsv-spy}
\end{figure}

\section{Results}
\label{sec:results}

\subsection{Computational setting}
Our method is implemented in Julia 1.11 \cite{julia}.
Optimization models are implemented in JuMP \cite{jump1},
using PowerModels \cite{powermodels} and PowerModelsSecurityConstrained \cite{powermodelssecurity}
for the SCOPF problem.
Neural network models are trained using PyTorch and embedded into the JuMP models using
MathOptAI \cite{dowson2025moai}.
KKT matrices are extracted from each optimization problem using
MadNLP \cite{shin2021graph,shin2024accelerating,madnlpgithub}
and we use block triangular and block-diagonal decompositions provided by MathProgIncidence \cite{parker2023dulmage}.
We compare our method against the HSL MA57 \cite{ma57} and MA86 \cite{ma86} linear solvers.
While MA86 can exploit parallel processing of the assembly tree and our algorithm can exploit
L3 BLAS (matrix-matrix operations) parallelism, we run all experiments with a single processor
in order to make a constant-hardware comparison. This is done by setting the
number of OMP threads (used by MA86) to one and with the following Julia options:
\\

\begin{center}\begin{minipage}{0.6\textwidth}
\begin{verbatim}
LinearAlgebra.BLAS.set_num_threads(1)
LinearAlgebra.BLAS.lbt_set_num_threads(1)
\end{verbatim}
\end{minipage}\end{center}

\medskip

These results were produced using the Selene supercomputer
at Los Alamos National Laboratory on a compute node with
two Intel 8468V 3.8 GHz processors and 2TB of RAM.
Code and instructions to reproduce these results can be found at
\url{https://github.com/Robbybp/moai-examples}.

\subsection{Numerical results}
\label{sec:num-results}

To evaluate our method for each combination of optimization problem and neural
network, we extract KKT matrices (and corresponding right-hand-sides) from the
first ten iterations of the MadNLP interior point method\footnote{
To make sure we compare different linear solvers on the same KKT systems,
we always compute MadNLP's search direction with our linear solver.}
and compare the performance of our method with that of MA57 and MA86 on these
linear systems.
The KKT matrices at each iteration have the same structure, so we only perform a
single symbolic factorization (referred to as ``initialization'').
Then we factorize the KKT matrix and backsolve with each of our ten KKT systems.
Backsolves are repeated in iterative refinement until
the max-norm of the residual is below $10^{-5}$.
\Cref{tab:runtime-summary} reports the total time required, broken down into
initialization, factorization, and backsolve steps for the ten linear
systems as well as the average residual max-norm and average number of
refinement iterations required.
These backsolve runtimes include all refinement iterations.
The speedup factor is a ratio of our solver's runtime in the factorization
plus backsolve phases to that of MA57 or MA86. We omit initialization time from this
calculation as this step is amortized over many iterations of an interior point
method.
We note that our method computes the same inertia as MA57 or MA86 in every case.
\begin{table}
  \centering
  \caption{Runtime results for ten KKT matrix factorizations for each model-solver
  combination. For each problem, we compare our method to a state-of-the-art alternative.}
  \resizebox{\textwidth}{!}{
  \small
  \begin{tabular}{cccccccccc}
\toprule
\multirow{2}{*}{Model} & \multirow{2}{*}{Solver} & \multirow{2}{*}{NN param.} & \multicolumn{3}{c}{Total runtime (s)} & \multicolumn{3}{c}{Average} \\
\cmidrule(lr){4-6}\cmidrule(lr){7-9}
& & & Init. & Fact. & Solve & Resid. & Refin. iter. & Speedup ($\times$)\\
\midrule
\multirow{3}{*}{MNIST} &   \multirow{3}{*}{MA57} &        1M &       0.04 &         2 &       0.1 &  2.7E-06 &                0 &      -- \\
                       &                         &        5M &        0.1 &        22 &       0.5 &  3.3E-07 &                0 &      -- \\
                       &                         &       18M &          1 &       125 &         2 &  1.3E-06 &                0 &      -- \\
\midrule
\multirow{3}{*}{MNIST} &   \multirow{3}{*}{Ours} &        1M &        0.5 &         6 &       0.4 &  2.4E-08 &                1 &    0.25 \\
                       &                         &        5M &          3 &        12 &         1 &  8.2E-07 &                2 &     1.7 \\
                       &                         &       18M &         10 &        31 &        13 &  3.1E-06 &                7 &     2.9 \\
\midrule
\multirow{3}{*}{SCOPF} &   \multirow{3}{*}{MA86} &      578k &       0.03 &         3 &       0.1 &  7.3E-07 &                1 &      -- \\
                       &                         &        4M &        0.2 &        32 &         1 &  2.9E-10 &                1 &      -- \\
                       &                         &       15M &        0.8 &       171 &         4 &  3.5E-10 &                1 &      -- \\
\midrule
\multirow{3}{*}{SCOPF} &   \multirow{3}{*}{Ours} &      578k &        0.2 &         1 &       0.1 &  1.9E-06 &                1 &     2.6 \\
                       &                         &        4M &          2 &         3 &       0.7 &  1.4E-06 &                1 &      10 \\
                       &                         &       15M &          7 &         9 &         3 &  2.5E-08 &                1 &      15 \\
\midrule
\multirow{3}{*}{LSV}   &   \multirow{3}{*}{MA86} &      111k &       0.01 &       0.3 &      0.03 &  1.1E-08 &                0 &      -- \\
                       &                         &      837k &       0.04 &         5 &       0.1 &  6.8E-07 &                0 &      -- \\
                       &                         &        9M &        0.5 &       149 &         1 &  1.3E-06 &                0 &      -- \\
\midrule
\multirow{3}{*}{LSV}   &   \multirow{3}{*}{Ours} &      111k &       0.06 &       0.6 &      0.04 &  1.4E-06 &                0 &    0.51 \\
                       &                         &      837k &        0.4 &         3 &       0.3 &  2.1E-06 &                2 &     1.6 \\
                       &                         &        9M &          6 &        14 &         3 &  2.9E-06 &                3 &     8.8 \\
\bottomrule
  \end{tabular}
}
  \label{tab:runtime-summary}
\end{table}

Our method achieves lower factorization times at the cost of more expensive initializations
and backsolves. Our faster factorizations make up for slower backsolves
and achieve overall speedups between 2.9$\times$ and 15$\times$ 
for the largest neural networks considered
even though more iterations of iterative refinement are sometimes required.

We hypothesize that our method outperforms MA57 and MA86
due to its limited fill-in when factorizing the block triangularizable pivot matrix.
As discussed in \Cref{sec:fillin}, we expect MA57 and MA86 to incur significant fill-in when factorizing the
original KKT matrix or the pivot matrix.
\Cref{tab:fill-in} shows that this is the case.
This table shows the size of the matrix factors and number of floating point
operations required by these solvers on the original, pivot, and Schur complement matrices.
For the original KKT and pivot matrices, MA57 and MA86 require factors 3-4$\times$ larger
than the original matrix. We attribute this fill-in partially to the relatively
small number of $2\times 2$ pivots selected by these algorithms. The number of $2\times 2$
pivots is equal to about 25\% of the matrix dimension, corresponding to about 50\%
of rows and columns in the factor matrices. Our analysis in \Cref{sec:fillin} indicates
that fill-in would be reduced if all of the pivots (at least in the pivot matrix)
were $2\times 2$.
\begin{table}
  \centering
  \caption{Factor sizes and floating point operations required to factorize KKT, pivot, and Schur complement
  matrices}
  \resizebox{\textwidth}{!}{
  \small
  \begin{tabular}{ccccccccc}
\toprule
Model & Solver & NN param. & Matrix & Matrix dim. &  NNZ & Factor NNZ & FLOPs & N. 2$\times$2 pivots \\
\midrule
\multirow{3}{*}{MNIST} &   \multirow{3}{*}{MA57} &       \multirow{3}{*}{18M} &    KKT &         44k &  18M &        65M &  177B &                  10k \\
                       &                         &                            &  Pivot &         41k &  16M &        50M &  138B &                  10k \\
                       &                         &                            &  Schur &          3k & 321k &       323k &  168M &                    0 \\
\midrule
\multirow{3}{*}{SCOPF} &   \multirow{3}{*}{MA86} &       \multirow{3}{*}{15M} &    KKT &         51k &  16M &        72M &  127B &                  17k \\
                       &                         &                            &  Pivot &         48k &  15M &        71M &  124B &                  17k \\
                       &                         &                            &  Schur &          3k &  16k &       104k &    5M &                   1k \\
\midrule
  \multirow{3}{*}{LSV} &   \multirow{3}{*}{MA86} &        \multirow{3}{*}{9M} &    KKT &         36k &   9M &        44M &  105B &                  13k \\
                       &                         &                            &  Pivot &         24k &   8M &        39M &   84B &                   7k \\
                       &                         &                            &  Schur &         11k & 213k &       588k &  136M &                   5k \\
\bottomrule
  \end{tabular}
}
  \label{tab:fill-in}
\end{table}

While our method already outperforms MA57 and MA86 running on identical
hardware, there is significant opportunity to further improve our implementation.
\Cref{tab:breakdown} shows a breakdown of runtime of the factorization
and backsolve phases of our implementation.
We first note that our factorization phase is dominated either
by the cost to factorize the pivot matrix, $C$,
or by the cost to build the Schur complement, i.e., performing the matrix subtraction, backsolve,
and multiplication necessary to compute $S=A-B^T C^{-1}B$.
These are the steps in lines \ref{alg:schur-factorize:fact-C} and \ref{alg:schur-factorize:build-S}
of \Cref{alg:schur-factorize}.
When building the Schur complement, backsolves involving $C$ (as described by
\Cref{eqn:bt-backsolve}) are the dominant computation cost.
While the blocks of the solution matrix (or vector), $X_i$,
need to be computed sequentially, the subtraction and multiplication operations
used to construct these blocks
(i.e., $B_i - \sum_{j=1}^{i-1}C_{ji}X_j$) can exploit single instruction, multiple
data (SIMD) parallelism on multiple CPU cores or a GPU.
Depending on the structure of $C_{ii}$ (i.e., if it is block-diagonal),
the backsolve may be parallelized as well.
Factorizing the individual diagonal blocks of the pivot matrix (i.e., $C_{ii}$)
can also be parallelized.
\begin{table}
  \centering
  \caption{Breakdown of solve times with our Schur complement solver}
  \small
  \begin{tabular}{cccccccccc}
\toprule
\multirow{2}{*}{Model}
& \multirow{2}{*}{NN param.}
& \multirow{2}{*}{Factorize (s)}
& \multicolumn{3}{c}{Percent of factorize time (\%)}
& \multirow{2}{*}{Backsolve (s)}
\\
\cmidrule{4-6}
&&& Build Schur & Schur & Pivot \\
\midrule
\multirow{3}{*}{MNIST} &        1M &         6 &          88 &               1 &               6 &       0.4 \\
                       &        5M &        13 &          76 &               1 &              17 &         2 \\
                       &       18M &        32 &          70 &               0 &              26 &        13 \\
\midrule
\multirow{3}{*}{SCOPF} &      578k &       0.4 &          46 &               2 &              44 &       0.1 \\
                       &        4M &         3 &          26 &               0 &              67 &       0.8 \\
                       &       15M &        10 &          25 &               0 &              68 &         3 \\
\midrule
\multirow{3}{*}{LSV}   &      111k &       0.6 &          72 &              14 &               3 &      0.05 \\
                       &      837k &         2 &          84 &               3 &               9 &       0.3 \\
                       &        9M &        15 &          65 &               0 &              29 &         3 \\
\bottomrule
  \end{tabular}
  \label{tab:breakdown}
\end{table}


\section{Conclusion}
We have shown that exploiting block triangular submatrices can yield
up to 15$\times$ performance improvements when solving KKT systems compared to
off-the-shelf multifrontal and supernodal methods. We have provided a theoretical
justification for this speed up based on limited fill-in achieved with our
method and have justified that the lack of inertia information from an unsymmetric
block triangular solver is not prohibitive to using this method in nonlinear
optimization methods. In fact, we have shown that our decomposition method will
fail if inertia correction is attempted in the block triangular matrix that
we exploit.

These results motivate several directions for future work.
First, the methods we have developed can be applied to many other optimization
problems with block triangular structures. For example, many sequential (dynamic
or multi-period) optimization problems have this block triangular structure.
Additionally, many problems that are not explicitly sequential
may have large block triangular submatrices that can be exploited
by our methods. We have shown in previous work that large triangular
submatrices can be automatically identified \cite{naik2025aggregation};
future work should use Schur complement decompositions to exploit these
submatrices in KKT systems. Additionally, the methods developed in
\cite{naik2025aggregation} may be relaxed to automatically identify
large \textit{block}-triangular submatrices.
Finally, a key advantage of our proposed method is that several of
the computational bottlenecks (backsolving with the pivot matrix
and constructing the Schur complement) can be parallelized on SIMD
hardware. This decomposition should be implemented on GPUs with
effort made to reduce CPU-GPU data transfer.

This work demonstrates that structure-exploiting decomposition algorithms
can accelerate linear system solves even in the absence of advanced hardware.
We intend this work to motivate research into additional structures that
can be exploited in KKT systems. We believe that there is a significant
opportunity improve performance of optimization algorithms through tighter
integration between optimization problems and linear solvers. Here, we propose
one method for improving performance in this way, although there is much
that can be done to further develop this method and others.

\section*{Acknowledgments}
This work was funded by the Los Alamos National Laboratory LDRD program
under the Artimis project, the Center for Nonlinear Studies, and an ISTI
Rapid Response award.

We thank Alexis Montoison for helpful conversations about our method, the
HSL libraries, and an earlier draft of this manuscript.

\appendix

\section{Size of submatrices}
\Cref{tab:matrix-size} contains the dimensions and number of (structural) nonzero
entries for KKT matrices and each submatrix considered in \Cref{eqn:kkt-partition}.
\begin{table}
  \centering
  \small
  \caption{Size of submatrices for each model and neural network}
  \begin{tabular}{cccccc}
\toprule
Model & NN param. & Matrix & N. row & N. col &  NNZ \\
\midrule
\multirow{5}{*}{MNIST} & \multirow{5}{*}{1M} &    KKT &    13k &    13k &   1M \\
                       &                     &    $A$ &     3k &     3k &   5k \\
                       &                     &    $B$ &    10k &     3k & 401k \\
                       &                     &  Pivot matrix ($C$) &    10k &    10k &   1M \\
                       &                     &  Schur complement ($S$) &     3k &     3k & 321k \\
\midrule
\multirow{5}{*}{MNIST} & \multirow{5}{*}{5M} &    KKT &    23k &    23k &   5M \\
                       &                     &    $A$ &     3k &     3k &   5k \\
                       &                     &    $B$ &    20k &     3k & 802k \\
                       &                     &  Pivot matrix ($C$) &    20k &    20k &   4M \\
                       &                     &  Schur complement ($S$) &     3k &     3k & 321k \\
\midrule
\multirow{5}{*}{MNIST} & \multirow{5}{*}{18M} &    KKT &    44k &    44k &  18M \\
                       &                      &    $A$ &     3k &     3k &   5k \\
                       &                      &    $B$ &    41k &     3k &   1M \\
                       &                      &  Pivot matrix ($C$) &    41k &    41k &  16M \\
                       &                      &  Schur complement ($S$) &     3k &     3k & 321k \\
\midrule
\multirow{5}{*}{SCOPF} & \multirow{5}{*}{578k} &    KKT &     9k &     9k & 578k \\
                       &                       &    $A$ &     3k &     3k &  10k \\
                       &                       &    $B$ &     6k &     3k &  39k \\
                       &                       &  Pivot matrix ($C$) &     6k &     6k & 526k \\
                       &                       &  Schur complement ($S$) &     3k &     3k &  16k \\
\midrule
\multirow{5}{*}{SCOPF} & \multirow{5}{*}{4M} &    KKT &    23k &    23k &   4M \\
                       &                     &    $A$ &     3k &     3k &  10k \\
                       &                     &    $B$ &    20k &     3k &  78k \\
                       &                     &  Pivot matrix ($C$) &    20k &    20k &   4M \\
                       &                     &  Schur complement ($S$) &     3k &     3k &  16k \\
\midrule
\multirow{5}{*}{SCOPF} & \multirow{5}{*}{15M} &    KKT &    51k &    51k &  16M \\
                       &                      &    $A$ &     3k &     3k &  10k \\
                       &                      &    $B$ &    48k &     3k & 117k \\
                       &                      &  Pivot matrix ($C$) &    48k &    48k &  15M \\
                       &                      &  Schur complement ($S$) &     3k &     3k &  16k \\
\midrule
\multirow{5}{*}{LSV}   & \multirow{5}{*}{111k} &    KKT &    13k &    13k & 142k \\
                       &                       &    $A$ &    11k &    11k &  28k \\
                       &                       &    $B$ &     1k &    11k &  54k \\
                       &                       &  Pivot matrix ($C$) &     1k &     1k &  58k \\
                       &                       &  Schur complement ($S$) &    11k &    11k & 213k \\
  \midrule
  \multirow{5}{*}{LSV} & \multirow{5}{*}{837k} &    KKT &    18k &    18k & 876k \\
                       &                       &    $A$ &    11k &    11k &  28k \\
                       &                       &    $B$ &     6k &    11k & 216k \\
                       &                       &  Pivot matrix ($C$) &     6k &     6k & 627k \\
                       &                       &  Schur complement ($S$) &    11k &    11k & 213k \\
  \midrule
  \multirow{5}{*}{LSV} & \multirow{5}{*}{9M} &    KKT &    36k &    36k &   9M \\
                       &                     &    $A$ &    11k &    11k &  28k \\
                       &                     &    $B$ &    24k &    11k & 866k \\
                       &                     &  Pivot matrix ($C$) &    24k &    24k &   8M \\
                       &                     &  Schur complement ($S$) &    11k &    11k & 213k \\
\bottomrule
  \end{tabular}
  \label{tab:matrix-size}
\end{table}

\section{Comparison of linear solvers}
To justify our use of MA57 and MA86, we evaluate the time required to solve
MadNLP's KKT system at the initial guess with the solvers
MA57 \cite{ma57}, MA86 \cite{ma86}, and MA97 \cite{ma97} from the HSL library.
The results are shown in Table \ref{tab:solver-comparison}.
Each solver uses default options other than the pivot order, which we
set to an approximate minimum degree (AMD) \cite{amestoy1996amd}
variant with efficient handling of dense rows \cite{amestoy2004dense}.
The results indicate that MA57 is the most performant solver for the MNIST test problem,
while MA86 is most performant for the SCOPF and LSV problems.
While MA86 and MA97 can exploit multicore parallelism, we run all solvers (including ours)
with a single thread to make a constant-hardware comparison.
\begin{table}
  \centering
  \caption{Runtime of three linear solvers on the KKT matrix at the initial guess}
  \small
  \begin{tabular}{ccccccc}
\toprule
Model & Solver & NN param. & Initialize & Factorize & Backsolve & Residual \\
\midrule
\multirow{9}[3]{*}{MNIST} &   \multirow{3}{*}{MA57} &        1M &       0.04 &       0.2 &      0.01 &  5.6E-10 \\
                       &                         &        5M &        0.2 &         2 &      0.05 &  9.5E-11 \\
                       &                         &       18M &        0.5 &        13 &       0.2 &  1.8E-11 \\
                       \cmidrule(ll){2-7}
                       &   \multirow{3}{*}{MA86} &        1M &       0.06 &       0.7 &      0.02 &  6.7E-16 \\
                       &                         &        5M &        0.2 &         5 &      0.09 &  1.1E-15 \\
                       &                         &       18M &          1 &        30 &       0.3 &  6.1E-15 \\
                       \cmidrule(ll){2-7}
                       &   \multirow{3}{*}{MA97} &        1M &       0.08 &         2 &      0.02 &  1.8E-15 \\
                       &                         &        5M &        0.4 &        15 &      0.08 &  9.1E-16 \\
                       &                         &       18M*&        --  &        -- &      --   &  --      \\
\midrule
\multirow{9}[3]{*}{SCOPF} &   \multirow{3}{*}{MA57} &      578k &       0.02 &       0.3 &  $< 0.01$ &  7.7E-09 \\
                          &                         &        4M &        0.1 &         5 &      0.05 &  5.9E-09 \\
                          &                         &       15M &        0.5 &        28 &       0.2 &  5.0E-09 \\
                          \cmidrule(ll){2-7}
                          &   \multirow{3}{*}{MA86} &      578k &       0.02 &       0.3 &      0.01 &  5.3E-09 \\
                          &                         &        4M &        0.2 &         3 &      0.06 &  2.1E-09 \\
                          &                         &       15M &        0.8 &        17 &       0.3 &  1.8E-09 \\
                          \cmidrule(ll){2-7}
                          &   \multirow{3}{*}{MA97} &      578k &       0.03 &         1 &      0.01 &  7.4E-09 \\
                          &                         &        4M &        0.3 &        11 &      0.06 &  2.9E-09 \\
                          &                         &       15M*&        --  &        -- &      --   &  --      \\
\midrule
\multirow{9}[3]{*}{LSV}   &   \multirow{3}{*}{MA57} &      111k &       0.01 &      0.03 &  $< 0.01$ &  8.7E-11 \\
                          &                         &      837k &        0.2 &       0.7 &      0.01 &  2.8E-12 \\
                          &                         &        9M$^\dag$&  --  &       --  &      --   &  --      \\
                          \cmidrule(ll){2-7}
                          &   \multirow{3}{*}{MA86} &      111k &       0.01 &      0.03 &  $< 0.01$ &  4.4E-09 \\
                          &                         &      837k &       0.04 &       0.5 &      0.01 &  2.6E-12 \\
                          &                         &        9M &        0.5 &        15 &       0.2 &  2.0E-12 \\
                          \cmidrule(ll){2-7}
                          &   \multirow{3}{*}{MA97} &      111k &       0.01 &      0.03 &  $< 0.01$ &  6.1E-11 \\
                          &                         &      837k &       0.05 &         1 &      0.01 &  4.5E-12 \\
                          &                         &       18M*&        --  &        -- &      --   &  --      \\
\bottomrule
  \multicolumn{7}{l}{\scriptsize $^*$Segmentation fault}\\[-2pt]
  \multicolumn{7}{l}{\scriptsize $^\dag$Int32 overflow}\\[-2pt]
  \end{tabular}
  \label{tab:solver-comparison}
\end{table}

\bibliographystyle{siamplain}
\bibliography{ref}
\end{document}


\maketitle

\section{Complete numerical results}
\label{sec:complete-runtime}
{\small
  \begin{longtable}{cccccccccccc}
    \caption{Complete numerical results} \\
\toprule
\multirow{2}{*}{Model} & \multirow{2}{*}{Solver} & \multirow{2}{*}{NN} & \multirow{2}{*}{Sample} & \multicolumn{3}{c}{Runtime (s)} & \multirow{2}{*}{Neg. eig.} & \multirow{2}{*}{Residual} & \multirow{2}{*}{Refinement iter.} & \multirow{2}{*}{Speedup}\\
\cmidrule{5-7}
& & & & Init. & Fact. & Solve & & \\
\midrule
\endfirsthead
\toprule
\multirow{2}{*}{Model} & \multirow{2}{*}{Solver} & \multirow{2}{*}{NN} & \multirow{2}{*}{Sample} & \multicolumn{3}{c}{Runtime (s)} & \multirow{2}{*}{Neg. eig.} & \multirow{2}{*}{Residual} & \multirow{2}{*}{Refinement iter.} & \multirow{2}{*}{Speedup}\\
\cmidrule{5-7}
& & & & Init. & Fact. & Solve & & \\
\midrule
\endhead
MNIST &   MA57 &        1M &      1 &       0.04 &       0.2 &      0.01 &         5936 &  5.6E-10 &                0 &      -- \\
MNIST &   MA57 &        1M &      2 &       0.04 &       0.2 &      0.01 &         5936 &  8.2E-06 &                0 &      -- \\
MNIST &   MA57 &        1M &      3 &       0.04 &       0.2 &      0.02 &         5936 &  4.1E-13 &                1 &      -- \\
MNIST &   MA57 &        1M &      4 &       0.04 &       0.2 &      0.01 &         5936 &  6.1E-06 &                0 &      -- \\
MNIST &   MA57 &        1M &      5 &       0.04 &       0.2 &      0.01 &         5936 &  2.9E-06 &                0 &      -- \\
MNIST &   MA57 &        1M &      6 &       0.04 &       0.2 &      0.01 &         5936 &  2.1E-06 &                0 &      -- \\
MNIST &   MA57 &        1M &      7 &       0.04 &       0.2 &      0.01 &         5936 &  1.8E-06 &                0 &      -- \\
MNIST &   MA57 &        1M &      8 &       0.04 &       0.2 &      0.01 &         5936 &  7.9E-07 &                0 &      -- \\
MNIST &   MA57 &        1M &      9 &       0.04 &       0.2 &      0.01 &         5936 &  1.2E-06 &                0 &      -- \\
MNIST &   MA57 &        1M &     10 &       0.04 &       0.2 &      0.01 &         5936 &  3.6E-06 &                0 &      -- \\
\midrule
MNIST &   MA57 &        5M &      1 &        0.1 &         2 &      0.05 &        11056 &  9.5E-11 &                0 &      -- \\
MNIST &   MA57 &        5M &      2 &        0.1 &         2 &      0.05 &        11056 &  1.9E-07 &                0 &      -- \\
MNIST &   MA57 &        5M &      3 &        0.1 &         2 &      0.05 &        11056 &  3.2E-07 &                0 &      -- \\
MNIST &   MA57 &        5M &      4 &        0.1 &         2 &      0.05 &        11056 &  2.2E-07 &                0 &      -- \\
MNIST &   MA57 &        5M &      5 &        0.1 &         2 &      0.05 &        11056 &  5.4E-07 &                0 &      -- \\
MNIST &   MA57 &        5M &      6 &        0.1 &         2 &      0.05 &        11056 &  2.0E-07 &                0 &      -- \\
MNIST &   MA57 &        5M &      7 &        0.1 &         2 &      0.05 &        11056 &  3.4E-07 &                0 &      -- \\
MNIST &   MA57 &        5M &      8 &        0.1 &         2 &      0.05 &        11056 &  4.8E-07 &                0 &      -- \\
MNIST &   MA57 &        5M &      9 &        0.1 &         2 &      0.05 &        11056 &  2.7E-07 &                0 &      -- \\
MNIST &   MA57 &        5M &     10 &        0.1 &         2 &      0.05 &        11056 &  7.4E-07 &                0 &      -- \\
\midrule
MNIST &   MA57 &       18M &      1 &          1 &        13 &       0.2 &        21296 &  1.8E-11 &                0 &      -- \\
MNIST &   MA57 &       18M &      2 &          1 &        12 &       0.2 &        21296 &  3.7E-07 &                0 &      -- \\
MNIST &   MA57 &       18M &      3 &          1 &        12 &       0.2 &        21296 &  8.1E-07 &                0 &      -- \\
MNIST &   MA57 &       18M &      4 &          1 &        13 &       0.2 &        21296 &  1.2E-06 &                0 &      -- \\
MNIST &   MA57 &       18M &      5 &          1 &        12 &       0.2 &        21296 &  4.1E-07 &                0 &      -- \\
MNIST &   MA57 &       18M &      6 &          1 &        13 &       0.2 &        21296 &  3.9E-07 &                0 &      -- \\
MNIST &   MA57 &       18M &      7 &          1 &        13 &       0.2 &        21296 &  2.1E-06 &                0 &      -- \\
MNIST &   MA57 &       18M &      8 &          1 &        13 &       0.2 &        21296 &  3.5E-07 &                0 &      -- \\
MNIST &   MA57 &       18M &      9 &          1 &        13 &       0.2 &        21296 &  1.5E-06 &                0 &      -- \\
MNIST &   MA57 &       18M &     10 &          1 &        13 &       0.2 &        21296 &  6.0E-06 &                0 &      -- \\
\midrule
MNIST &   Ours &        1M &      1 &        0.5 &       0.5 &      0.03 &         5936 &  1.1E-07 &                0 &    0.29 \\
MNIST &   Ours &        1M &      2 &        0.5 &       0.5 &      0.04 &         5936 &  2.3E-08 &                1 &    0.28 \\
MNIST &   Ours &        1M &      3 &        0.5 &       0.5 &      0.04 &         5936 &  3.1E-08 &                1 &     0.3 \\
MNIST &   Ours &        1M &      4 &        0.5 &       0.5 &      0.04 &         5936 &  2.9E-08 &                1 &    0.29 \\
MNIST &   Ours &        1M &      5 &        0.5 &         1 &      0.04 &         5936 &  1.7E-08 &                1 &    0.16 \\
MNIST &   Ours &        1M &      6 &        0.5 &       0.5 &      0.04 &         5936 &  1.0E-08 &                1 &     0.3 \\
MNIST &   Ours &        1M &      7 &        0.5 &       0.5 &      0.04 &         5936 &  4.3E-10 &                1 &    0.29 \\
MNIST &   Ours &        1M &      8 &        0.5 &         1 &      0.04 &         5936 &  1.4E-09 &                1 &    0.15 \\
MNIST &   Ours &        1M &      9 &        0.5 &       0.5 &      0.04 &         5936 &  4.5E-09 &                1 &     0.3 \\
MNIST &   Ours &        1M &     10 &        0.5 &       0.5 &      0.04 &         5936 &  1.2E-08 &                1 &    0.29 \\
\midrule
MNIST &   Ours &        5M &      1 &          3 &         1 &       0.1 &        11056 &  2.6E-07 &                1 &     1.8 \\
MNIST &   Ours &        5M &      2 &          3 &         1 &       0.2 &        11056 &  1.1E-06 &                2 &     1.7 \\
MNIST &   Ours &        5M &      3 &          3 &         1 &       0.2 &        11056 &  5.5E-07 &                2 &     1.7 \\
MNIST &   Ours &        5M &      4 &          3 &         1 &       0.2 &        11056 &  1.1E-06 &                2 &     1.7 \\
MNIST &   Ours &        5M &      5 &          3 &         1 &       0.2 &        11056 &  1.2E-06 &                2 &     1.7 \\
MNIST &   Ours &        5M &      6 &          3 &         1 &       0.2 &        11056 &  2.4E-07 &                2 &     1.7 \\
MNIST &   Ours &        5M &      7 &          3 &         1 &       0.2 &        11056 &  4.6E-08 &                2 &     1.7 \\
MNIST &   Ours &        5M &      8 &          3 &         1 &       0.2 &        11056 &  5.0E-09 &                2 &     1.7 \\
MNIST &   Ours &        5M &      9 &          3 &         1 &       0.1 &        11056 &  2.1E-06 &                1 &     1.8 \\
MNIST &   Ours &        5M &     10 &          3 &         1 &       0.1 &        11056 &  1.6E-06 &                1 &     1.8 \\
\midrule
MNIST &   Ours &       18M &      1 &         10 &         3 &         2 &        21296 &  9.1E-06 &                9 &     2.4 \\
MNIST &   Ours &       18M &      2 &         10 &         4 &         2 &        21296 &  2.9E-06 &                9 &     2.5 \\
MNIST &   Ours &       18M &      3 &         10 &         3 &         2 &        21296 &  2.4E-06 &               10 &     2.5 \\
MNIST &   Ours &       18M &      4 &         10 &         3 &         2 &        21296 &  4.7E-06 &                9 &     2.8 \\
MNIST &   Ours &       18M &      5 &         10 &         3 &         1 &        21296 &  3.3E-06 &                8 &     2.9 \\
MNIST &   Ours &       18M &      6 &         10 &         3 &         1 &        21296 &  1.4E-06 &                7 &     3.0 \\
MNIST &   Ours &       18M &      7 &         10 &         3 &         1 &        21296 &  5.6E-06 &                5 &     3.3 \\
MNIST &   Ours &       18M &      8 &         10 &         3 &         1 &        21296 &  9.4E-07 &                5 &     3.3 \\
MNIST &   Ours &       18M &      9 &         10 &         3 &       0.8 &        21296 &  5.7E-07 &                4 &     3.4 \\
MNIST &   Ours &       18M &     10 &         10 &         3 &       0.8 &        21296 &  1.1E-07 &                4 &     3.3 \\
\midrule
SCOPF &   MA86 &      578k &      1 &       0.03 &       0.3 &      0.01 &         4534 &  5.3E-09 &                0 &      -- \\
SCOPF &   MA86 &      578k &      2 &       0.03 &       0.3 &      0.01 &         4534 &  9.6E-12 &                1 &      -- \\
SCOPF &   MA86 &      578k &      3 &       0.03 &       0.3 &      0.02 &         4534 &  1.4E-11 &                1 &      -- \\
SCOPF &   MA86 &      578k &      4 &       0.03 &       0.3 &      0.01 &         4534 &  1.3E-11 &                1 &      -- \\
SCOPF &   MA86 &      578k &      5 &       0.03 &       0.3 &      0.01 &         4534 &  7.3E-06 &                0 &      -- \\
SCOPF &   MA86 &      578k &      6 &       0.03 &       0.3 &      0.01 &         4534 &  3.5E-11 &                1 &      -- \\
SCOPF &   MA86 &      578k &      7 &       0.03 &       0.3 &      0.01 &         4534 &  4.7E-11 &                1 &      -- \\
SCOPF &   MA86 &      578k &      8 &       0.03 &       0.3 &      0.01 &         4534 &  6.7E-11 &                1 &      -- \\
SCOPF &   MA86 &      578k &      9 &       0.03 &       0.3 &      0.01 &         4534 &  7.1E-11 &                1 &      -- \\
SCOPF &   MA86 &      578k &     10 &       0.03 &       0.3 &      0.01 &         4534 &  8.4E-11 &                1 &      -- \\
\midrule
SCOPF &   MA86 &        4M &      1 &        0.2 &         3 &      0.06 &        11534 &  2.1E-09 &                0 &      -- \\
SCOPF &   MA86 &        4M &      2 &        0.2 &         3 &       0.1 &        11534 &  1.8E-11 &                1 &      -- \\
SCOPF &   MA86 &        4M &      3 &        0.2 &         3 &       0.1 &        11534 &  3.7E-11 &                1 &      -- \\
SCOPF &   MA86 &        4M &      4 &        0.2 &         3 &       0.1 &        11534 &  3.2E-11 &                1 &      -- \\
SCOPF &   MA86 &        4M &      5 &        0.2 &         3 &       0.1 &        11534 &  5.8E-11 &                1 &      -- \\
SCOPF &   MA86 &        4M &      6 &        0.2 &         3 &       0.1 &        11534 &  6.6E-11 &                1 &      -- \\
SCOPF &   MA86 &        4M &      7 &        0.2 &         3 &       0.1 &        11534 &  8.9E-11 &                1 &      -- \\
SCOPF &   MA86 &        4M &      8 &        0.2 &         3 &       0.1 &        11534 &  1.1E-10 &                1 &      -- \\
SCOPF &   MA86 &        4M &      9 &        0.2 &         3 &       0.1 &        11534 &  1.6E-10 &                1 &      -- \\
SCOPF &   MA86 &        4M &     10 &        0.2 &         3 &       0.1 &        11534 &  2.2E-10 &                1 &      -- \\
\midrule
SCOPF &   MA86 &       15M &      1 &        0.8 &        17 &       0.3 &        25534 &  1.8E-09 &                0 &      -- \\
SCOPF &   MA86 &       15M &      2 &        0.8 &        17 &       0.4 &        25534 &  1.2E-10 &                1 &      -- \\
SCOPF &   MA86 &       15M &      3 &        0.8 &        17 &       0.4 &        25534 &  5.8E-11 &                1 &      -- \\
SCOPF &   MA86 &       15M &      4 &        0.8 &        17 &       0.4 &        25534 &  7.3E-11 &                1 &      -- \\
SCOPF &   MA86 &       15M &      5 &        0.8 &        17 &       0.5 &        25534 &  1.4E-10 &                1 &      -- \\
SCOPF &   MA86 &       15M &      6 &        0.8 &        17 &       0.4 &        25534 &  1.1E-10 &                1 &      -- \\
SCOPF &   MA86 &       15M &      7 &        0.8 &        17 &       0.4 &        25534 &  9.3E-11 &                1 &      -- \\
SCOPF &   MA86 &       15M &      8 &        0.8 &        17 &       0.4 &        25534 &  4.5E-10 &                1 &      -- \\
SCOPF &   MA86 &       15M &      9 &        0.8 &        17 &       0.4 &        25534 &  2.5E-10 &                1 &      -- \\
SCOPF &   MA86 &       15M &     10 &        0.8 &        17 &       0.5 &        25534 &  4.6E-10 &                1 &      -- \\
\midrule
SCOPF &   Ours &      578k &      1 &        0.2 &      0.04 &      0.01 &         4534 &  4.1E-09 &                0 &     5.8 \\
SCOPF &   Ours &      578k &      2 &        0.2 &      0.04 &      0.01 &         4534 &  4.0E-06 &                0 &     5.7 \\
SCOPF &   Ours &      578k &      3 &        0.2 &      0.04 &      0.01 &         4534 &  9.3E-06 &                0 &     5.3 \\
SCOPF &   Ours &      578k &      4 &        0.2 &       0.5 &      0.01 &         4534 &  5.2E-06 &                0 &    0.49 \\
SCOPF &   Ours &      578k &      5 &        0.2 &      0.04 &      0.01 &         4534 &  4.5E-11 &                1 &     5.1 \\
SCOPF &   Ours &      578k &      6 &        0.2 &      0.05 &      0.01 &         4534 &  3.4E-10 &                1 &     4.5 \\
SCOPF &   Ours &      578k &      7 &        0.2 &      0.05 &      0.01 &         4534 &  3.6E-08 &                1 &     4.7 \\
SCOPF &   Ours &      578k &      8 &        0.2 &      0.04 &      0.01 &         4534 &  3.2E-08 &                1 &     4.8 \\
SCOPF &   Ours &      578k &      9 &        0.2 &      0.04 &      0.01 &         4534 &  8.2E-09 &                1 &     4.9 \\
SCOPF &   Ours &      578k &     10 &        0.2 &      0.04 &      0.01 &         4534 &  8.9E-08 &                1 &     5.0 \\
\midrule
SCOPF &   Ours &        4M &      1 &          2 &       0.3 &      0.06 &        11534 &  1.2E-09 &                0 &     9.7 \\
SCOPF &   Ours &        4M &      2 &          2 &       0.2 &      0.06 &        11534 &  6.9E-06 &                0 &      11 \\
SCOPF &   Ours &        4M &      3 &          2 &       0.3 &      0.05 &        11534 &  6.1E-06 &                0 &      11 \\
SCOPF &   Ours &        4M &      4 &          2 &       0.2 &      0.08 &        11534 &  1.7E-09 &                1 &      10 \\
SCOPF &   Ours &        4M &      5 &          2 &       0.2 &      0.08 &        11534 &  3.4E-09 &                1 &      10 \\
SCOPF &   Ours &        4M &      6 &          2 &       0.2 &      0.08 &        11534 &  5.3E-09 &                1 &      10 \\
SCOPF &   Ours &        4M &      7 &          2 &       0.2 &      0.08 &        11534 &  1.6E-08 &                1 &      10 \\
SCOPF &   Ours &        4M &      8 &          2 &       0.3 &      0.08 &        11534 &  2.9E-08 &                1 &     7.8 \\
SCOPF &   Ours &        4M &      9 &          2 &       0.2 &      0.08 &        11534 &  2.1E-07 &                1 &      10 \\
SCOPF &   Ours &        4M &     10 &          2 &       0.2 &      0.08 &        11534 &  3.9E-07 &                1 &      10 \\
\midrule
SCOPF &   Ours &       15M &      1 &          7 &         1 &       0.2 &        25534 &  1.6E-09 &                0 &      15 \\
SCOPF &   Ours &       15M &      2 &          7 &       0.9 &       0.3 &        25534 &  1.2E-10 &                1 &      15 \\
SCOPF &   Ours &       15M &      3 &          7 &       0.9 &       0.3 &        25534 &  6.3E-11 &                1 &      15 \\
SCOPF &   Ours &       15M &      4 &          7 &       0.9 &       0.3 &        25534 &  7.0E-09 &                1 &      15 \\
SCOPF &   Ours &       15M &      5 &          7 &       0.9 &       0.3 &        25534 &  7.9E-09 &                1 &      15 \\
SCOPF &   Ours &       15M &      6 &          7 &         1 &       0.3 &        25534 &  6.0E-09 &                1 &      15 \\
SCOPF &   Ours &       15M &      7 &          7 &       0.9 &       0.4 &        25534 &  1.8E-09 &                1 &      14 \\
SCOPF &   Ours &       15M &      8 &          7 &       0.9 &       0.3 &        25534 &  5.5E-08 &                1 &      15 \\
SCOPF &   Ours &       15M &      9 &          7 &       0.9 &       0.3 &        25534 &  4.9E-08 &                1 &      15 \\
SCOPF &   Ours &       15M &     10 &          7 &         1 &       0.3 &        25534 &  1.2E-07 &                1 &      15 \\
\midrule
  LSV &   MA86 &      111k &      1 &       0.01 &      0.03 &  $< 0.01$ &         4307 &  4.4E-09 &                0 &      -- \\
  LSV &   MA86 &      111k &      2 &       0.01 &      0.03 &  $< 0.01$ &         4308 &  1.3E-08 &                0 &      -- \\
  LSV &   MA86 &      111k &      3 &       0.01 &      0.03 &  $< 0.01$ &         4308 &  2.0E-08 &                0 &      -- \\
  LSV &   MA86 &      111k &      4 &       0.01 &      0.03 &  $< 0.01$ &         4308 &  2.3E-08 &                0 &      -- \\
  LSV &   MA86 &      111k &      5 &       0.01 &      0.03 &  $< 0.01$ &         4308 &  1.7E-08 &                0 &      -- \\
  LSV &   MA86 &      111k &      6 &       0.01 &      0.03 &  $< 0.01$ &         4308 &  1.0E-08 &                0 &      -- \\
  LSV &   MA86 &      111k &      7 &       0.01 &      0.03 &  $< 0.01$ &         4308 &  1.0E-08 &                0 &      -- \\
  LSV &   MA86 &      111k &      8 &       0.01 &      0.03 &  $< 0.01$ &         4308 &  5.2E-09 &                0 &      -- \\
  LSV &   MA86 &      111k &      9 &       0.01 &      0.03 &  $< 0.01$ &         4308 &  3.0E-09 &                0 &      -- \\
  LSV &   MA86 &      111k &     10 &       0.01 &      0.03 &  $< 0.01$ &         4308 &  1.2E-10 &                0 &      -- \\
  \midrule
  LSV &   MA86 &      837k &      1 &       0.04 &       0.5 &      0.01 &         6611 &  2.6E-12 &                0 &      -- \\
  LSV &   MA86 &      837k &      2 &       0.04 &       0.5 &      0.01 &         6612 &  3.9E-07 &                0 &      -- \\
  LSV &   MA86 &      837k &      3 &       0.04 &       0.5 &      0.01 &         6612 &  2.2E-07 &                0 &      -- \\
  LSV &   MA86 &      837k &      4 &       0.04 &       0.5 &      0.01 &         6612 &  4.9E-07 &                0 &      -- \\
  LSV &   MA86 &      837k &      5 &       0.04 &       0.5 &      0.02 &         6612 &  6.0E-07 &                0 &      -- \\
  LSV &   MA86 &      837k &      6 &       0.04 &       0.5 &      0.01 &         6612 &  6.0E-07 &                0 &      -- \\
  LSV &   MA86 &      837k &      7 &       0.04 &       0.5 &      0.01 &         6612 &  8.4E-07 &                0 &      -- \\
  LSV &   MA86 &      837k &      8 &       0.04 &       0.5 &      0.01 &         6612 &  6.1E-07 &                0 &      -- \\
  LSV &   MA86 &      837k &      9 &       0.04 &       0.5 &      0.01 &         6612 &  1.5E-06 &                0 &      -- \\
  LSV &   MA86 &      837k &     10 &       0.04 &       0.5 &      0.01 &         6612 &  1.5E-06 &                0 &      -- \\
  \midrule
  LSV &   MA86 &        9M &      1 &        0.5 &        15 &       0.2 &        15827 &  2.0E-12 &                0 &      -- \\
  LSV &   MA86 &        9M &      2 &        0.5 &        15 &       0.1 &        15828 &  4.8E-07 &                0 &      -- \\
  LSV &   MA86 &        9M &      3 &        0.5 &        15 &       0.1 &        15828 &  3.3E-07 &                0 &      -- \\
  LSV &   MA86 &        9M &      4 &        0.5 &        15 &       0.1 &        15828 &  3.8E-07 &                0 &      -- \\
  LSV &   MA86 &        9M &      5 &        0.5 &        15 &       0.2 &        15828 &  3.9E-07 &                0 &      -- \\
  LSV &   MA86 &        9M &      6 &        0.5 &        15 &       0.1 &        15828 &  6.9E-07 &                0 &      -- \\
  LSV &   MA86 &        9M &      7 &        0.5 &        15 &       0.1 &        15828 &  1.3E-06 &                0 &      -- \\
  LSV &   MA86 &        9M &      8 &        0.5 &        15 &       0.2 &        15828 &  2.2E-06 &                0 &      -- \\
  LSV &   MA86 &        9M &      9 &        0.5 &        15 &       0.1 &        15828 &  3.2E-06 &                0 &      -- \\
  LSV &   MA86 &        9M &     10 &        0.5 &        15 &       0.1 &        15828 &  4.5E-06 &                0 &      -- \\
  \midrule
  LSV &   Ours &      111k &      1 &       0.06 &      0.06 &  $< 0.01$ &         4307 &  3.8E-11 &                0 &    0.46 \\
  LSV &   Ours &      111k &      2 &       0.06 &      0.05 &      0.01 &         4308 &  1.4E-07 &                1 &    0.51 \\
  LSV &   Ours &      111k &      3 &       0.06 &      0.06 &      0.01 &         4308 &  2.9E-07 &                1 &    0.49 \\
  LSV &   Ours &      111k &      4 &       0.06 &      0.06 &      0.01 &         4308 &  2.4E-07 &                1 &    0.44 \\
  LSV &   Ours &      111k &      5 &       0.06 &      0.05 &  $< 0.01$ &         4308 &  7.0E-06 &                0 &    0.53 \\
  LSV &   Ours &      111k &      6 &       0.06 &      0.05 &  $< 0.01$ &         4308 &  2.2E-06 &                0 &    0.54 \\
  LSV &   Ours &      111k &      7 &       0.06 &      0.05 &  $< 0.01$ &         4308 &  1.6E-06 &                0 &    0.54 \\
  LSV &   Ours &      111k &      8 &       0.06 &      0.06 &  $< 0.01$ &         4308 &  1.0E-06 &                0 &    0.53 \\
  LSV &   Ours &      111k &      9 &       0.06 &      0.06 &  $< 0.01$ &         4308 &  1.1E-06 &                0 &    0.53 \\
  LSV &   Ours &      111k &     10 &       0.06 &      0.06 &  $< 0.01$ &         4308 &  6.7E-07 &                0 &    0.52 \\
  \midrule
  LSV &   Ours &      837k &      1 &        0.4 &       0.2 &      0.01 &         6611 &  1.7E-10 &                0 &     2.1 \\
  LSV &   Ours &      837k &      2 &        0.4 &       0.2 &      0.02 &         6612 &  7.6E-07 &                1 &     2.0 \\
  LSV &   Ours &      837k &      3 &        0.4 &       0.2 &      0.02 &         6612 &  1.2E-06 &                1 &     2.0 \\
  LSV &   Ours &      837k &      4 &        0.4 &       0.2 &      0.02 &         6612 &  2.9E-06 &                1 &     2.1 \\
  LSV &   Ours &      837k &      5 &        0.4 &       0.9 &      0.02 &         6612 &  2.1E-06 &                1 &    0.58 \\
  LSV &   Ours &      837k &      6 &        0.4 &       0.2 &      0.02 &         6612 &  1.9E-06 &                1 &     2.1 \\
  LSV &   Ours &      837k &      7 &        0.4 &       0.2 &      0.03 &         6612 &  5.3E-07 &                2 &     2.0 \\
  LSV &   Ours &      837k &      8 &        0.4 &       0.2 &      0.03 &         6612 &  9.8E-07 &                2 &     2.0 \\
  LSV &   Ours &      837k &      9 &        0.4 &       0.2 &      0.03 &         6612 &  8.0E-06 &                2 &     2.0 \\
  LSV &   Ours &      837k &     10 &        0.4 &       0.2 &      0.05 &         6612 &  3.0E-06 &                4 &     1.9 \\
  \midrule
  LSV &   Ours &        9M &      1 &          6 &         1 &       0.1 &        15827 &  1.8E-10 &                0 &     9.6 \\
  LSV &   Ours &        9M &      2 &          6 &         1 &       0.2 &        15828 &  1.3E-06 &                1 &     9.6 \\
  LSV &   Ours &        9M &      3 &          6 &         1 &       0.2 &        15828 &  1.9E-06 &                1 &     9.5 \\
  LSV &   Ours &        9M &      4 &          6 &         1 &       0.3 &        15828 &  5.5E-07 &                2 &     9.1 \\
  LSV &   Ours &        9M &      5 &          6 &         1 &       0.3 &        15828 &  3.4E-06 &                2 &     9.0 \\
  LSV &   Ours &        9M &      6 &          6 &         1 &       0.3 &        15828 &  4.5E-06 &                2 &     9.0 \\
  LSV &   Ours &        9M &      7 &          6 &         1 &       0.3 &        15828 &  8.6E-07 &                3 &     8.7 \\
  LSV &   Ours &        9M &      8 &          6 &         1 &       0.4 &        15828 &  2.0E-06 &                4 &     8.4 \\
  LSV &   Ours &        9M &      9 &          6 &         1 &       0.5 &        15828 &  8.4E-06 &                5 &     7.9 \\
  LSV &   Ours &        9M &     10 &          6 &         1 &       0.7 &        15828 &  6.4E-06 &                7 &     7.4 \\
\bottomrule
    \label{tab:complete-runtime}
\end{longtable}
}

\bibliographystyle{siamplain}
\bibliography{ref}